\documentclass[preprint,12pt]{elsarticle}

\biboptions{sort&compress}

\usepackage{amsmath,amsthm,amssymb}
\usepackage{mathrsfs}
\usepackage{color}

\def\Rset{\mathbb{R}}
\def\Cset{\mathbb{C}}
\def\Kset{\mathbb{K}}
\def\Nset{\mathbb{N}}

\def\fsep{\mathrm{sep}(f)}

\newtheorem{theorem}{Theorem}[section]
\newtheorem{lemma}[theorem]{Lemma}
\newtheorem{corollary}[theorem]{Corollary}
\newtheorem{proposition}[theorem]{Proposition}

\theoremstyle{definition}  
\newtheorem{definition}[theorem]{Definition}
\newtheorem{remark}[theorem]{Remark}
\newtheorem{example}[theorem]{Example}

\begin{document}

\begin{frontmatter}
\title{General convergence theorems for iterative processes and applications to the Weierstrass root-finding method}
\author{Petko~D.~Proinov}
\ead{proinov@uni-plovdiv.bg}
\address{Faculty of Mathematics and Informatics, University of Plovdiv, Plovdiv 4000, Bulgaria}

\begin{abstract}
In this paper, we prove some general convergence theorems for the Picard iteration in cone metric spaces over a solid vector space.
As an application, we provide a detailed convergence analysis of the Weierstrass iterative method for computing 
all zeros of a polynomial simultaneously. These results improve and generalize existing ones in the literature.
\end{abstract}


\numberwithin{equation}{section}    

\end{frontmatter}



\section{Introduction}
\label{sec:Introduction}

In the first part of the paper, we study the convergence of the iterative processes of the type
\begin{equation}\label{eq:Picard-iteration}
x_{n + 1}  = Tx_n ,  \quad n = 0,1,2,\ldots,
\end{equation}
where $T \colon D \subset X \to X$ is an iteration function in a cone metric space $(X,d)$ over a solid vector space ${(Y,\preceq)}$.
Cone metric spaces have a long history 
(see Collatz \cite{Col66}, 
Zabrejko \cite{Zab97}, 
Jankovi\'c, Kadelburg and Radenovi\'c \cite{JKR11}, 
Proinov \cite{Pro13} and references therein).
For an overview of the theory of cone metric spaces over a solid vector space, we refer the reader to
\cite{Pro13} and \cite[Section~2]{PN14}.

In the second part of the paper, we study the convergence of the famous Weierstrass method \cite{Wei91} for computing all zeros of a polynomial simultaneously. This method was introduced and studied for the first time by Weierstrass in 1891. 
In 1960--1966, the method was rediscovered by 
Durand \cite{Dur60} (in implicit form),
Dochev \cite{Doc62b},
Kerner \cite{Ker66}
and
Pre{\v s}i\'c \cite{Pre66}.
For this reason, it is also known as `Durand-Kerner method', `Weierstrass-Dochev method', etc.
For an overview of iterative methods for simultaneous finding of polynomial zeros, we refer the reader to
\cite{SAK94,Kyu98,McN07,Pet08}.
 
Throughout this paper, ${(\Kset,|\cdot|)}$ denotes an arbitrary normed (valued) field with absolute value ${|\cdot|}$,  
and $\Kset[z]$ denotes denotes the ring of polynomials in one variable $z$ over $\Kset$. 
Let $f \in \Kset[z]$ be a polynomial of degree $n \ge 2$. 
We consider the zeros of $f$ as a vector in $\Kset^n$.
More precisely, a vector ${\xi \in \Kset^n}$ is said to be a \emph{root-vector} of $f$ if
\begin{equation} \label{eq:root-vector}
f(z) = a_0 \prod _{i = 1} ^ n (z - \xi_i) \quad\text{for all } \, z \in \Kset,
\end{equation}
where ${a_0 \in \Kset}$. Obviously, $f$ has a root-vector in $\Kset^n$ if and only if $f$ splits in $\Kset$.
Recall that the Weierstrass method is defined by the following iteration
\begin{equation}  \label{eq:Weierstrass-iteration}
x^{k + 1}  = x^k - W(x^k ), \qquad k = 0,1,2,\ldots,
\end{equation}
where ${W \colon \mathcal{D} \subset \Kset^n \to \Kset^n}$ is defined by
${W(x) = (W_1(x),\ldots,W_n(x))}$ with
\begin{equation} \label{eq:Weierstrass-correction}
W_i(x) = \frac{f(x_i)}{a_0 \displaystyle\prod_{j \ne i} (x_i  - x_j)}
\qquad (i = 1,\ldots,n),
\end{equation}
where $a_0$ is the leading coefficient of $f$ and $\mathcal{D}$ is the set of all vectors in $\Kset^n$ with distinct components.
The operator $W$ is called the \emph{Weierstrass correction}.
Sometimes we write $W_f$ instead of $W$ to indicate that the operator $W$ is generated  by $f$.
It is easy to see that the Weierstrass correction $W_f$ is invariant with respect to multiplication of $f$
by a non-zero constant $c \in \Kset$.
Obviously, the Weierstrass iteration \eqref{eq:Weierstrass-iteration} can be represented in the form \eqref{eq:Picard-iteration} 
with the  iteration function ${T \colon \mathcal{D} \subset \Kset^n \to \Kset^n}$ defined by
\begin{equation} \label{Weierstrass-iteration-function}
	T(x) = x - W(x).
\end{equation}

The aim of this paper is twofold.
First, we present some general convergence theorems with error estimates for the Picard iteration 
\eqref{eq:Picard-iteration}. These results extend some of the results in \cite{Pro09,Pro10}.
Second, using these results we provide a detailed convergence analysis of the Weierstrass method \eqref{eq:Picard-iteration}.
The new results for the Weierstrass method improve the corresponding results of 
\cite{Bat98,Doc62b,Pre80,Zhe82,KM83,Zhe87,WZ91,ZW93,WZ95,PCT95,Pet96,PHI98,Til98,Yak02,Han00,PH01,Pro06b,PP13}. 

The paper is structured as follows:

In Section~\ref{sec:Preliminaries}, we present some preliminaries results and notations that will be useful in the sequel.

In Section~\ref{sec:Local-convergence-theorems-in-cone-metric-spaces}, we establish two general convergence theorems with error estimates for iterated contractions at a point in cone metric spaces. 
The first one extends Theorem~3.6 of \cite{Pro09}.

In Section~\ref{sec:General-convergence-theorems-in-cone-metric-spaces}, we establish two general semilocal convergence theorems with error estimates for iterative processes of the type \eqref{eq:Picard-iteration}.
These results extend Theorems 5.4 and 5.6 of \cite{Pro10}.
As a consequence we obtain a convergence theorem with error estimates for iterated contractions in cone metric spaces, 
which extends Theorem~6.5 of \cite{Pro10}. 
All results in this sections are generalizations of the Banach Contraction Principle \cite{Ban22} as well as of the Iterated Contraction Principle given in \cite[Chap.~12]{OR70} and \cite{Pro13}.

In Section~\ref{sec:Some-inequalities}, we present some inequalities in $\Kset^n$ and notations which will be useful in the next sections.

In Section~\ref{sec:Local-convergence-of-the-Weierstrass-method-I}, we obtain a local convergence theorem with error estimates for the Weierstrass method which improves the results of
Dochev \cite{Doc62b},
Kyurkchiev and Markov \cite{KM83}, 
Yakoubsohn \cite{Yak02} and
Proinov and Petkova \cite{PP13}.

In Section~\ref{sec:Local-convergence-of-the-Weierstrass-method-II}, we obtain another local convergence theorem 
with error estimates for the Weierstrass method which improves and generalizes the results of
Wang and Zhao \cite{WZ91}, 
Tilli \cite{Til98} and 
Han \cite{Han00}.

In Section~\ref{sec:Semilocal-convergence-of-the-Weierstrass-method}, we prove a new convergence theorem for 
the Weierstrass method under computationally verifiable initial conditions.
The main result of this section generalizes, improves and complements all previous results in this area, which are due to 
Pre{\v s}i\'c \cite{Pre80},
Zheng \cite{Zhe82,Zhe87},
Wang and Zhao \cite{ZW93,WZ95},
Petkovi\'c, Carstensen and Trajkovi\'c \cite{PCT95},
Petkovi\'c \cite{Pet96},
Petkovi\'c, Herceg and Ili\'c \cite{PHI98},
Batra \cite{Bat98},
Han \cite{Han00},
Petkovi\'c and Herceg \cite{PH01} and
Proinov \cite{Pro06b}.
The new result in this section also gives computationally verifiable error estimates,
a localization formula for the Weierstrass iteration \eqref{eq:Weierstrass-iteration} as well as 
a sufficient condition for a polynomial ${f \in \Kset[z]}$ of degree ${n \ge 2}$ to have $n$ simple zeros in the field $\Kset$. 
Finally, we provide an example which shows the exactness of the error estimates of our semilocal theorem for the Weierstrass iterative method.


\section{Preliminaries}
\label{sec:Preliminaries}

Throughout the paper, $J$ denotes an interval in ${\Rset}_+$ containing $0$, that is, an interval of the form
${[0, R]}$, ${[0, R)}$ or ${[0, \infty)}$, where ${R > 0}$ . We use the abbreviation $\varphi^n$ for the $n$th iterate of a function 
$\varphi \colon J \to J$.
For $n \in \Nset$ we denote by $S_n (t)$ the following polynomial
\begin{equation} \label{eq:sum-of-geometric-progression}
S_n (t) = \sum_{k = 0}^{n - 1} {t^k} .
\end{equation}
If the case ${n = 0}$, we set ${S_0 (t) \equiv 0}$.
Throughout the paper, we assume by definition that $0^0 = 1$.

\begin{definition}[\cite{Pro10}] \label{def:quasi-homogeneous-function}
A function ${\varphi \colon J \to \Rset_+}$ is called \emph{quasi-homogeneous} of degree ${r \ge 0}$ on $J$ if it satisfies the following condition
\begin{equation} \label{eq:quasi-homogeneous-function-1}
\varphi(\lambda t) \le \lambda^r \varphi(t)
\quad\text{for all } \lambda  \in [0,1] \text{ and } t  \in J.
\end{equation}
\end{definition}

It is easy to prove that a function ${\varphi \colon J \to \Rset_+}$ is quasi-homogeneous of degree $r \ge 0$ on $J$ if and only if there exists a nondecreasing function ${\Phi \colon J \to \Rset_+}$ such that ${\varphi(t) = t^r \, \Phi(t)}$ for all ${t \in J}$. 

Let us give an example for quasi-homogeneous functions, which we will use in 
Sections \ref{sec:Local-convergence-of-the-Weierstrass-method-I} and \ref{sec:Local-convergence-of-the-Weierstrass-method-II}.

\begin{example} \label{exmp:quasi-homogeneous-function-1}
Let ${n \in \Nset}$ and $\varphi$ be a quasi-homogeneous function of degree ${r > 0}$ on an interval $J$.
Then the function $\Phi$ defined by
\[
\Phi(t) = (1 + \varphi(t))^n - 1 ,
\]
is also quasi-homogeneous of degree $r$ on $J$.
\end{example}

\begin{proof}
It follows from the identity
\(
\Phi(t) = \sum_{k=1}^n {\binom{n}{k} \varphi(t)^k}
\)
because the sum of quasi-homogeneous functions of degree $r$ on $J$ is quasi-homogeneous of degree $r$ on $J$.
\end{proof}

\begin{definition}[\cite{Pro07}] \label{def:gauge-function-of-high-order}
A function $\varphi \colon J \to J$ is called a
\emph{gauge function of order} $r \ge 1$ on an interval $J$ if it is quasi-homogeneous of degree $r$ on $J$ and
\begin{equation} \label{eq:gauge-function-of-high-order-second-condition}
\varphi(t) \le t \quad\text{for all } t \in J.
\end{equation}
A gauge function $\varphi$ of order $r$ on $J$ is said to be a \emph{strict gauge function} if
the inequality in \eqref{eq:gauge-function-of-high-order-second-condition} holds strictly whenever ${t > 0}$. 
\end{definition}

\begin{proposition}[\cite{Pro10}]  \label{prop:gauge-function-sufficient-condition}
If ${\varphi \colon J \to \Rset_+}$ is quasi-homogeneous of degree ${r \ge 1}$ on an interval $J$ and
${R > 0}$ is fixed point of $\varphi$ in $J$, then  $\varphi$ is a gauge function of order $r$ on $[0,R]$. Moreover, if ${r > 1}$, then
$\varphi$ is a strict gauge function of order
$r$ on ${[0,R)}$.
\end{proposition}

\begin{definition}[\cite{Pro09}]  \label{def:function-of-initial-condition}
Let ${T \colon D \subset X \to X}$ be a map of a set $X$. A function $E \colon D \to {\Rset}_+ $ is said to be \emph{function of initial conditions} of $T$ (with a gauge function $\varphi $ on $J$) if there exists a function $\varphi \colon J \to J$ such that
\[ 
E(Tx) \le \varphi(E(x)) \, \text{ for all } x \in D \text{ with } Tx \in D \text{ and } E(x) \in J.
\] 
\end{definition}

\begin{definition}[\cite{Pro09}] \label{def:initial-point-E}
Let $T \colon D \subset X \to X$ be a map of a set $X$ and
$E \colon D \to {\Rset}_+ $ be a function of initial conditions of $T$ with a gauge function on $J$. 
Then a point $x \in D$ is said to be an \emph{initial point} of $T$ if
$E(x) \in J$ and ${T^n x \in D}$ for all ${n \ge 0}$. 
\end{definition}

\begin{proposition}[\cite{Pro10}] \label{prop:initial-point-test}
Let $T \colon D \subset X \to X$ be a map of a set $X$ and
$E \colon D \to {\Rset}_+$ be a function of initial conditions of  $T$ with a gauge function $\varphi$ on $J$. Suppose
\begin{equation} \label{eq:initial-point-test}
x \in D \text{ with } E(x) \in J \text{ implies } Tx \in D.	
\end{equation}
Then every point $x_0 \in D$ such that $E(x_0) \in J$ is an initial point of $T$.
\end{proposition}

\begin{proposition}[\cite{Pro09}] \label{prop:initial-point-properties}
Let $T \colon D \subset X \to X$ be a map of a set $X$ and $E \colon D \to {\Rset}_+ $ be a function of initial conditions of $T$ with a
gauge function $\varphi$ on an interval $J$.
If $x \in D$ is an initial point of $T$, then every iterate $x_n =T^n x$ $(n = 0,1,2,\ldots)$ is an initial point of $T$.
Besides, if $\varphi$ is a gauge function of order $r \ge 1$, then for all ${n \ge 0}$,
\[
E(x_{n+1}) \le  \lambda^{r^n} E(x_n) \quad\text{and}\quad E(x_n) \le E(x_0) \, {\lambda}^{S_n (r)},
\]
where $\lambda=\phi(E(x_0))$ and ${\phi \colon J \to [0,1]}$ is a nondecreasing function such that 
\begin{equation} \label{eq:standard-condition-phi}
\varphi(t) = t \, \phi(t)
\quad\text{for all } \, t \in J.
\end{equation}
\end{proposition}

The following definition extends Definition~3.4 of \cite{Pro10}. For an ordered vector space ${(Y,\preceq)}$ we denote by $Y_+$ the positive cone of $Y$, that is,
\[
Y_+ = \{x \in Y : x \succeq 0 \}.
\]

\begin{definition} \label{def:convergence-function-CMS}
Let $(X,d)$ be a cone metric space over an ordered vector space ${(Y,\preceq)}$.
Suppose ${T \colon D \subset X \to X}$ is an operator and ${E \colon D \to \Rset_+}$ is a function of initial conditions of $T$
with gauge function $\varphi$ on an interval $J$.
Then a function $F \colon D \to Y_+$ is said to be a \emph{convergence function} of $T$ (with respect to $E$ and with control functions
$\beta$ and $\gamma$) if there exist two nondecreasing functions
$\beta \colon J \to [0, 1)$ and $\gamma \colon J \to \Rset_+$ such that
\begin{equation} \label{eq:convergence-function-CMS-condition-1}
F(Tx) \preceq \beta(E(x)) \, F(x) \, \text{ for all } \, x \in D \text{ with } Tx \in D \text{ and } E(x) \in J,
\end{equation}
\begin{equation} \label{eq:convergence-function-CMS-condition-2}
d(x, Tx) \preceq \gamma(E(x)) \, F(x) \, \text{ for all } \, x \in D \text{ with } E(x) \in J.
\end{equation}
\end{definition}

Usually we need stronger conditions for the control function $\beta$.
Namely, we assume that ${\beta \colon J \to [0,1)}$ is a nondecreasing function satisfying the following two conditions:
\begin{equation}  \label{eq:standard-condition-beta-1}
t \, \beta(t) \text{ is a strict gauge function of order } r \text{ on } J,
\end{equation}
\begin{equation}  \label{eq:standard-condition-beta-2}
\forall \, t \in J:  \phi(t) = 0 \text{ implies } \beta(t) = 0 ,
\end{equation}
where ${\phi \colon J \to [0,1]}$ is a nondecreasing function satisfying  \eqref{eq:standard-condition-phi}.

It is easy to see  that condition \eqref{eq:standard-condition-beta-2} is equivalent to the existence of a function
${\psi \colon J \to \Rset_+}$ such that
\begin{equation} \label{eq:standard-condition-psi} 
	\beta(t) = \phi(t) \, \psi(t) \quad\text{for all } \,\, t \in J.
\end{equation}

The following proposition extends Lemmas 3.5 and 3.8 of \cite{Pro10}. 

\begin{proposition} \label{prop:convergence-function-properties}
Let $T \colon D \subset X \to X$ be an operator of a cone metric space $(X,d)$ over an ordered vector space ${(Y,\preceq)}$,
$E \colon D \to {\Rset}_+$ be a function of initial conditions of $T$ with gauge function $\varphi$ on $J$ satisfying 
\eqref{eq:gauge-function-of-high-order-second-condition},
and let $F \colon D \to Y_+$ be a convergence function of $T$ with control functions $\beta$ and $\gamma$.
If $x \in D$ is an initial point of $T$, then for all $n \ge 0$,
\begin{equation} \label{eq:convergence-function-property-1}
F(x_{n+1}) \preceq h \, F(x_n)
\quad\text{and}\quad
F(x_n) \preceq h^n F(x_0),
\end{equation}
where $x_n = T^n x$ and $h = \beta(E(x_0))$. Besides, if $\varphi$ is a gauge function of order $r \ge 1$ on $J$ and $\beta$ satisfies 
\eqref{eq:standard-condition-beta-1} and \eqref{eq:standard-condition-beta-2}, then for all $n \ge 0$,
\begin{equation} \label{eq:convergence-function-property-2}
F(x_{n+1}) \preceq \theta \lambda^{r^n} F(x_n)
\quad\text{and}\quad
F(x_n) \preceq \theta^n \lambda^{S_n(r)} F(x_0),
\end{equation}
where $\lambda = \phi(E(x_0))$, $\theta = \psi(E(x_0))$ and
$\psi \colon J \to {\Rset}_+$ is a function satisfying \eqref{eq:standard-condition-psi}.
\end{proposition}

\begin{proof}
We shall prove only \eqref{eq:convergence-function-property-2} since the proof of \eqref{eq:convergence-function-property-1} is similar.
Suppose $x_0$ is an initial point of $T$. We have ${x_n \in D}$ and ${E(x_n) \in J}$ since every iterate $x_n$ is an initial point of $T$.
Setting ${x = x_n}$ in \eqref{eq:convergence-function-CMS-condition-1}, we obtain  
\begin{equation}  \label{eq:convergence-function-iterates}
F(x_{n+1}) \preceq \beta(E(x_n)) \, F(x_n).
\end{equation}
Condition \eqref{eq:standard-condition-beta-1} implies that $\beta$ is a quasi-homogeneous function of degree ${r-1}$ on the interval $J$.
From this and Proposition~\ref{prop:initial-point-properties}, we get
\begin{equation} \label{eq:property-beta}
\beta(E(x_n)) \le \beta(E(x_0) \lambda^{S_n(r)}) \le \lambda^{(r - 1) S_n(r)} \beta(E(x_0) =
\theta  \lambda^{1 + (r - 1) S_n(r)} = \theta  \lambda^{r^n} .
\end{equation}
Combining \eqref{eq:convergence-function-iterates} and \eqref{eq:property-beta}, we get the first inequality in
\eqref{eq:convergence-function-property-2}.
The second inequality in \eqref{eq:convergence-function-property-2} follows from the first one by induction on ${n \ge 0}$.
\end{proof}

We end this section with two proposition from the theory of cone metric spaces over a solid vector space.
These statements are trivial in the case of metric spaces.

\begin{proposition}[\cite{Pro13}] \label{prop:property-convergence}
Let ${(X,d)}$ be a cone metric space over a solid vector space
${(Y,\preceq)}$.
Suppose ${(x_n)}$ is a sequence in $X$ satisfying
\begin{equation} \label{eq:property-convergence}
d(x_n,x) \preceq b_n
\quad \text{for all }\, n \ge 0,
\end{equation}
where $x$ is a point in $X$ and
${(b_n)}$ is a sequence in $Y$ converging to $0$.
Then the sequence $(x_n)$ converges to $x$.
\end{proposition}

\begin{proposition}[\cite{Pro13}] \label{prop:properties-Cauchy-CMS}
Let ${(X,d)}$ be a cone metric space over a solid vector space ${(Y,\preceq)}$.
Suppose ${(x_n)}$ is a sequence in $X$ satisfying
\begin{equation} \label{eq:properties-Cauchy-CMS}
d(x_n,x_m) \preceq b_n \quad\text{for all \,} n, m \ge 0 \text{ with } m \ge n,
\end{equation}
where ${(b_n)}$ is a sequence in $Y$ which converges to $0$. Then:
\begin{enumerate}[(i)]
    \item The sequence ${(x_n)}$ is a Cauchy sequence in $X$.
    \item If ${(x_n)}$ converges to a point ${x \in X}$, then
\begin{equation} \label{eq:properties-Cauchy-error-estimate-CMS}
d(x_n,x) \preceq b_n \quad \text{for all }\, n \ge 0.
\end{equation}
\end{enumerate}
\end{proposition}


\section{Local convergence theorems in cone metric spaces}
\label{sec:Local-convergence-theorems-in-cone-metric-spaces}

In this section, we present two general convergence theorems with error estimates for iterated contractions at a point in cone metric spaces. 

\begin{definition}[Iterated Contraction at a Point]  \label{def:iterated-contraction-at-a-point}
Let ${T \colon D \subset X \to X}$ be an operator of a cone metric space ${(X,d)}$ over a solid vector space ${(Y,\preceq)}$, 
and let ${E \colon D \to {\Rset}_+}$ be a function of initial conditions of $T$ with a gauge function on an interval $J$.
Then $T$ is said to be an \emph{iterated contraction with respect to $E$ at a point} $\xi \in D$ 
(with control function $\beta$) if ${E(\xi) \in J}$ and
\begin{equation} \label{eq:local-iterated-contraction-E}
d(Tx,\xi) \preceq \beta(E(x)) \, d(x,\xi)
\quad\text{for all } x \in D \text{ with } E(x) \in J,
\end{equation}
where ${\beta \colon J \to [0,1)}$ is a nondecreasing function. 
\end{definition}

\begin{proposition} \label{lem:fixed-point}
Let ${T \colon D \subset X \to X}$ be an operator of a cone metric space ${(X,d)}$ over a solid vector space ${(Y,\preceq)}$, 
and let ${E \colon D \to {\Rset}_+}$ be a function of initial conditions of $T$ with a gauge function on an interval $J$.
If $T$ is an \emph{iterated contraction} with respect to $E$ at a point $\xi \in D$, then 
$\xi$ is a unique fixed point of $T$ in the set ${U = \{ x \in D : E(x) \in J \}}$.
\end{proposition}

\begin{proof}
Setting ${x = \xi}$ in \eqref{eq:local-iterated-contraction-E}, we get ${d(T\xi,\xi) \preceq 0}$. Therefore,
${d(T\xi,\xi) = 0}$ which means that $\xi$ is a fixed point of $T$. Suppose ${\eta \in U}$ is also a fixed point of $T$. 
Applying \eqref{eq:local-iterated-contraction-E} with
${x = \eta}$ we obtain
\(
d(\eta,\xi) \preceq \beta(E(\eta)) \, d(\eta,\xi)
\)
which implies that ${d(\eta,\xi) \preceq 0}$ because the values of $\beta$ are less than 1.  
From this, we conclude that ${\eta = \xi}$. Hence, $\xi$ is a unique fixed point of $T$ in the set $U$.
\end{proof}

In the following two theorems, we consider the problem of approximating the fixed points of iterated contractions at a point 
in a cone metric space. The first one extends Theorem~3.6 of \cite{Pro09}.

\begin{theorem} 
\label{thm:first-local-convergence-theorem-CMS}
Let ${T \colon D \subset X \to X}$ be an operator of a cone metric space ${(X,d)}$ over a solid vector space
${(Y,\preceq)}$, and let
${E \colon D \to {\Rset}_+}$ be a function of initial conditions of $T$ with a gauge function $\varphi$ of order $r$ on an interval $J$.
Suppose $T$ is an iterated contraction with respect to $E$ at a point $\xi$ 
with control function $\beta$ satisfying \eqref{eq:standard-condition-beta-1} and \eqref{eq:standard-condition-beta-2}.
Then for each initial point $x_0$ of $T$ the Picard iteration \eqref{eq:Picard-iteration}
remains in in the set ${U = \{ x \in D : E(x) \in J \}}$ and converges to $\xi$ with error estimates
\begin{equation}  \label{eq:first-local-convergence-theorem-CMS-error-estimates}
d(x_{n+1}, \xi) \preceq \theta \lambda^{r^n} d(x_n,\xi)
\quad\text{and}\quad
d(x_n,\xi) \preceq \theta^n \lambda^{S_n(r)} d(x_0,\xi)
\end{equation}
for all $n \ge 0$, where ${\lambda = \phi(E(x_0))}$, ${\theta  = \psi(E(x_0))}$ and
${\psi \colon J \to {\Rset}_+}$ is a function satisfying \eqref{eq:standard-condition-psi}. 
\end{theorem}

\begin{proof}
It follows from the second estimate in \eqref{eq:first-local-convergence-theorem-CMS-error-estimates} and 
Proposition~\ref{prop:property-convergence} with
\[ 
b_n = \theta^n \lambda ^{S_n(r)} d(x_0,\xi)
\] 
that $(x_n)$ converges to $\xi$ in $X$. 
By the definition of initial points, it follows that ${x_0 \in U}$. Starting from $x_0$ the iterative sequence
\eqref{eq:Picard-iteration} remains in the set $U$ because every iterate $x_n$ is an initial point of $T$.
Setting ${x = x_n}$ in \eqref{eq:local-iterated-contraction-E}, we obtain  
\begin{equation}  \label{eq:local-iterated-contraction-corollary}
d(x_{n+1}, \xi) \preceq \beta(E(x_n)) \, d(x_n,\xi),
\end{equation}
From \eqref{eq:local-iterated-contraction-corollary} and \eqref{eq:property-beta}, we get the first estimate in
\eqref{eq:first-local-convergence-theorem-CMS-error-estimates}.
The second estimate in \eqref{eq:first-local-convergence-theorem-CMS-error-estimates} is a consequence of the first one.
\end{proof}

Setting $\beta = \phi$ in Theorem~\ref{thm:first-local-convergence-theorem-CMS} we get the following result.

\begin{corollary} 
\label{thm:first-local-convergence-theorem-CMS-phi}
Let ${T \colon D \subset X \to X}$ be an operator of a cone metric space ${(X,d)}$ over a solid vector space
${(Y,\preceq)}$, and let ${E \colon D \to {\Rset}_+}$ be a function of initial conditions of $T$ with a strict gauge function $\varphi$ of order $r$ on an interval $J$. 
If $T$ is an iterated contraction with respect to $E$ at a point $\xi$ 
with control function $\phi$ satisfying \eqref{eq:standard-condition-phi}, then
for each initial point $x_0$ of $T$ the Picard iteration \eqref{eq:Picard-iteration}
remains in in the set ${U = \{ x \in D : E(x) \in J \}}$ and converges to $\xi$ with error estimates
\begin{equation}  \label{eq:first-local-convergence-theorem-CMS-phi-error-estimates}
d(x_{n+1}, \xi) \preceq \lambda^{r^n} \, d(x_n,\xi)
\quad\text{and}\quad
d(x_n,\xi) \preceq \lambda ^{S_n(r)} \, d(x_0,\xi)
\end{equation}
for all $n \ge 0$, where ${\lambda = \phi(E(x_0))}$.
\end{corollary}

\begin{theorem} 
\label{thm:third-local-convergence-theorem-CMS}
Let ${T \colon D \subset X \to X}$ be an operator of a cone metric space ${(X,d)}$ over a solid vector space
${(Y,\preceq)}$, and let ${E \colon D \to {\Rset}_+}$ be a function of initial conditions of $T$ with a nondecreasing gauge function 
$\varphi$ on an interval $J$.
Suppose $T$ is an iterated contraction with respect to $E$ at a point $\xi$ with control function $\beta$.
Assume there exist ${\sigma \in (0,1)}$, ${r \ge 1}$ and a nondecreasing function ${c \colon [0,\sigma] \to J}$ such that
\begin{equation}  \label{eq:third-local-convergence-theorem-CMS}
\beta(c(t)) \le t
\quad\text{and}\quad
\varphi(c(t)) \le c(t^r)
\quad\text{for all } \, t \in [0,\sigma].
\end{equation}
If $x_0$ is an initial point point of $T$ satisfying
\begin{equation} \label{eq:third-local-convergence-theorem-CMS-initial-condition}
E(x_0) \le c(\sigma),
\end{equation}
then Picard sequence \eqref{eq:Picard-iteration} converges to $\xi$ with error estimates
\begin{equation}  \label{eq:third-local-convergence-theorem-CMS-error-estimates}
d(x_{n+1},\xi) \preceq \sigma^{r^n} \, d(x_n,\xi)
\quad\text{and}\quad
d(x_n,\xi) \preceq \sigma^{S_n(r)} \, d(x_0,\xi)
\end{equation}
\end{theorem}

\begin{proof}
We shall prove only the first estimate in \eqref{eq:third-local-convergence-theorem-CMS-error-estimates}.  
Using  Definition~\ref{def:function-of-initial-condition} and the second inequality in \eqref{eq:third-local-convergence-theorem-CMS} 
it is easy to see that for every initial point $x$ of $T$ and ${t \in [0,\sigma]}$,
\[
E(x) \le c(t) \quad\text{implies}\quad E(Tx) \le c(t^r).
\]
Then by induction one can prove that
\[
E(x_n) \le c\left( \sigma^{r^n} \right)
\quad\text{for all}\quad n \ge 0.
\]
This inequality together with the first inequality in \eqref{eq:third-local-convergence-theorem-CMS} implies
\[
\beta(E(x_n)) \le \sigma^{r^n}.
\]
From this and \eqref{eq:local-iterated-contraction-corollary}, we get the first estimate in 
\eqref{eq:third-local-convergence-theorem-CMS-error-estimates}.
\end{proof}


\section{General convergence theorems in cone metric spaces}
\label{sec:General-convergence-theorems-in-cone-metric-spaces}

In this section, we establish two general convergence theorems with error estimates for iterative processes 
of the type \eqref{eq:Picard-iteration}. These results extend Theorems 5.4 and 5.6 of \cite{Pro10}.

\begin{theorem} 
\label{thm:first-convergence-theorem-CMS}
Let ${T \colon D \subset X \to X}$ be an operator of a complete cone metric space ${(X,d)}$ over a solid vector space ${(Y,\preceq)}$,
${E \colon D \to {\Rset}_+}$ be a function of initial conditions of $T$ with gauge function $\varphi$ on an interval $J$ satisfying
\eqref{eq:gauge-function-of-high-order-second-condition}, and let ${F \colon D \to Y_+}$ be a convergence function of $T$ with control functions $\beta$ and $\gamma$.
Then the following statements hold true:
\begin{enumerate}[(i)]
\item \textsc{Convergence}. Starting from any initial point $x_0$ of $T$, the Picard iteration \eqref{eq:Picard-iteration} is well-defined, remains in the closed ball ${\overline{U}(x_0,\rho)}$ and converges to a point ${\xi \in \overline{U}(x_0,\rho)}$, where
\[
\rho = \frac{\gamma(E(x_0))}{1 - \beta(E(x_0))} \, F(x_0) .
\]
\item \textsc{A priori estimate}. For all ${n \ge 0}$ we have the following estimate
\begin{equation} \label{eq:first-convergence-theorem-CMS-a-priori-estimate}
d(x_n,\xi) \preceq \frac{h^n}{1 - h} \, \gamma(E(x_0)) \, F(x_0) \,,
\end{equation}
where ${h = \beta(E(x_0))}$.
\item \textsc{First a posteriori estimate}. For all ${n \ge 0}$ we have the following estimate
\begin{equation} \label{eq:first-convergence-theorem-CMS-a-posteriori-estimate-1}
d(x_n,\xi) \preceq \frac{\gamma(E(x_n))}{1 - \beta(E(x_n))} \, F(x_n).
\end{equation}
\item \textsc{Second a posteriori estimate}. For all ${n \ge 0}$ we have the following estimate
\begin{equation} \label{eq:first-convergence-theorem-CMS-a-posteriori-estimate-2}
d(x_{n+1},\xi) \preceq \frac{\beta(E(x_n))}{1 - \beta(E(x_{n+1}))} \, \gamma(E(x_{n+1})) \, F(x_n) .
\end{equation}
\item
\textsc{Existence of a fixed point}. If ${\xi \in D}$ and $T$ is continuous at $\xi$, then $\xi$ is a fixed point of $T$.
\end{enumerate}
\end{theorem}

\begin{proof}
It follows from \eqref{eq:gauge-function-of-high-order-second-condition} and Proposition~\ref{prop:initial-point-properties} 
that ${E(x_n) \le E(x_0)}$ for all ${n \ge 0}$. 
Let ${m,n \in \Nset}$ with ${m \ge n}$.
From the triangle inequality, \eqref{eq:convergence-function-CMS-condition-2} and Proposition~\ref{prop:convergence-function-properties}, we get
\begin{eqnarray*}
d(x_n,x_m)
& \preceq & \sum_{j=n}^m {d(x_j,x_{j+1})} \preceq \sum_{j=n}^m {\gamma(E(x_j)) \, F(x_j)} \preceq \gamma(E(x_0)) \sum_{j=n}^m {F(x_j)}\\
& \preceq & \left( \sum_{j=n}^{m} {h^j} \right) \gamma(E(x_n)) F(x_0)
\preceq \frac{h^n}{1 - h} \, \gamma(E(x_0)) \, F(x_0).
\end{eqnarray*}
Therefore,
\begin{equation}\label{eq:general-Cauchy-2}
d(x_n,x_m) \preceq b_n, \quad\text{where}\quad b_n = \frac{h^n}{1 - h} \, \gamma(E(x_0)) \, F(x_0).
\end{equation}
Note that ${(b_n)}$ is a sequence in $Y$ which converges to $0$ since ${h^n \to 0}$ in $\Rset$.
By Proposition~\ref{prop:properties-Cauchy-CMS}(i), we conclude that ${(x_n)}$ is a Cauchy sequence.
By the completeness of the space $X$, we deduce that ${(x_n)}$ converges to a point ${\xi \in X}$.
Now it follows from Proposition~\ref{prop:properties-Cauchy-CMS}(ii) that
\begin{equation} \label{eq:a-priori-estimate-CMS}
d(x_n,\xi) \preceq b_n
\end{equation}
for every $n \ge 0$. From this we get the estimate \eqref{eq:first-convergence-theorem-CMS-a-priori-estimate}.
Setting $n=0$ in \eqref{eq:first-convergence-theorem-CMS-a-priori-estimate}, we get
\begin{equation} \label{eq:a-posteriori-estimate-CMS}
d(x_0,\xi) \preceq \frac{\gamma(E(x_0))}{1 - \beta(E(x_0))} \, F(x_0)
\end{equation}
which means that ${\xi \in \overline{U}(x_0,\rho)}$. The inequality \eqref{eq:a-posteriori-estimate-CMS} holds for every initial point
$x_0$ of $T$.
Therefore, applying \eqref{eq:a-posteriori-estimate-CMS} to $x_n$ instead of $x_0$, we obtain 
\eqref{eq:first-convergence-theorem-CMS-a-posteriori-estimate-1}.
From \eqref{eq:first-convergence-theorem-CMS-a-posteriori-estimate-1} and \eqref{eq:convergence-function-CMS-condition-1}, we obtain
\begin{eqnarray*}
d(x_{n+1},\xi)
& \preceq & \frac{\gamma(E(x_{n+1}))}{1 - \beta(E(x_{n+1}))} \, F(x_{n+1})\\
& \preceq & \frac{\beta(E(x_n))}{1 - \beta(E(x_{n+1}))} \, \gamma(E(x_{n+1})) \, F(x_n)
\end{eqnarray*}
which proves the estimate \eqref{eq:first-convergence-theorem-CMS-a-posteriori-estimate-2}.
Setting $n = 0$ in \eqref{eq:general-Cauchy-2}, we get
${d(x_0,x_m) \preceq \rho}$ for every ${m \ge 0}$. Hence, the sequence $(x_n)$ lies in the ball ${\overline{U}(x_0,\rho)}$.
\end{proof}

\begin{theorem} 
\label{thm:third-convergence-theorem-CMS}
Let ${(X,d)}$ be a complete cone metric space over a solid vector space ${(Y,\preceq)}$, ${T \colon D \subset X \to X}$
be an operator, ${E \colon D \to \Rset_+}$ be a function of initial conditions of $T$ with a gauge function
$\varphi$ of order ${r \ge 1}$ on an interval $J$, and let
${F \colon D \to Y_+}$ be a convergence function of $T$ with control functions $\beta$ and $\gamma$ satisfying
\eqref{eq:standard-condition-beta-1} and \eqref{eq:standard-condition-beta-2}.
Then the following statements hold true:
\begin{enumerate}[(i)]
\item
\textsc{Convergence}. Starting from any initial point $x_0$ of $T$, the Picard iteration \eqref{eq:Picard-iteration} is well-defined, 
remains in the closed ball ${\overline{U}(x_0,\rho)}$ and converges to a point ${\xi \in \overline{U}(x_0,\rho)}$, where
\[
\rho = \frac{\gamma(E(x_0))}{1 - \beta(E(x_0))} \, F(x_0).
\]
\item
\textsc{A priori estimate}. For all ${n \ge 0}$ we have the following error estimate
\begin{equation} \label{eq:third-convergence-theorem-CMS-a-priori-estimate}
d(x_n, \xi) \preceq \frac{\theta^n \lambda^{S_n(r)}}{1 - \theta \, \lambda^{r^n}} \, \gamma\left(E(x_0) \lambda^{S_n(r)}\right) F(x_0),
\end{equation}
where
$\lambda = \phi(E(x_0))$, $\theta  = \psi(E(x_0))$, ${\phi \colon J \to [0,1]}$ is a nondecreasing function satisfying
\eqref{eq:standard-condition-phi} and $\psi \colon J \to {\Rset}_+$ is a function satisfying \eqref{eq:standard-condition-psi}.
\item
\textsc{First a posteriori estimate}. For all ${n \ge 0}$ we have the following error estimate
\begin{equation} \label{eq:third-convergence-theorem-CMS-a-posteriori-estimate-1}
d(x_n, \xi) \preceq \frac{\gamma(E(x_n))}{1 - \beta(E(x_n))} \, F(x_n).
\end{equation}
\item
\textsc{Second a posteriori estimate}. For all ${n \ge 0}$ we have the following error estimate
\begin{equation} \label{eq:third-convergence-theorem-CMS-a-posteriori-estimate-2}
d(x_{n+1}, \xi) \preceq \frac{\theta_n \lambda_n}{1 - \theta_n (\lambda_n)^r} \, \gamma(E(x_{n + 1})) \, F(x_n) .
\end{equation}
where ${\lambda_n = \phi(E(x_n))}$ and ${\theta_n = \psi(E(x_n))}$.
\item \textsc{Some other estimates}. For all ${n \ge 0}$ we have
\begin{equation} \label{eq:third-convergence-theorem-CMS-some-other-estimates-1}
F(x_{n + 1}) \preceq \theta \, \lambda^{r^n} \, F(x_n)
\quad\text{and}\quad 
F(x_n) \preceq  \theta^n \, \lambda^{S_n(r)} \, F(x_0).
\end{equation}
\item \textsc{Existence of a fixed point}. If $\xi \in D$ and $T$ is continuous at $\xi$, then $\xi$ is a fixed point of $T$.
\end{enumerate}
\end{theorem}

\begin{proof}
Conclusions (i), (iii) and (vi)
follow immediately from Theorem~\ref{thm:first-convergence-theorem-CMS}.
Conclusion (ii) follows from \eqref{eq:third-convergence-theorem-CMS-a-posteriori-estimate-1}, 
Proposition~\ref{prop:initial-point-properties}, 
Proposition~\ref{prop:convergence-function-properties} and inequality \eqref{eq:property-beta}.
Conclusion (v) follows from Proposition~\ref{prop:convergence-function-properties}.
It remains to prove (iv).
It follows from \eqref{eq:standard-condition-psi} that 
\begin{equation} \label{eq:beta-product-1}
\beta(E(x_n)) = \theta_n \, \lambda_n .
\end{equation}
On the other hand, taking into account that ${E(x_{n+1}) \le \varphi(E(x_n))}$ and that $\beta$ is quasi-homogeneous 
of degree ${r-1}$ on $J$, we get
\begin{equation} \label{eq:beta-product-2}
\beta(E(x_{n+1}) \le \beta(\varphi(E(x_n)) = \beta(\lambda_n \, E(x_n)) \le (\lambda_n)^{r-1} \beta(E(x_n)) = \theta_n (\lambda_n)^r .
\end{equation}
Now conclusion (iv) follows from \eqref{eq:first-convergence-theorem-CMS-a-posteriori-estimate-2}, \eqref{eq:beta-product-1} and 
\eqref{eq:beta-product-2}.
\end{proof}

\begin{definition}[Iterated Contraction \cite{Pro10}]  \label{def:iterated-contraction-E}
Let ${T \colon D \subset X \to X}$ be an operator of a cone metric space ${(X,d)}$ over a solid vector space ${(Y,\preceq)}$, 
and let ${E \colon D \to {\Rset}_+}$ be a function of initial conditions of $T$ with a gauge function on an interval $J$.
Then $T$ is said to be an \emph{iterated contraction with respect to} $E$ (with control function $\beta$) if 
\[ 
d(Tx,T^2x) \preceq \beta(E(x)) \, d(x,Tx)
\quad\text{for all } x \in D \text{ with } x \in D \text{ and } E(x) \in J,
\]
where ${\beta \colon J \to [0,1)}$ is a nondecreasing function. 
\end{definition}

Setting ${F(x) = d(x,Tx)}$ in Theorem~\ref{thm:third-convergence-theorem-CMS}, we get the following convergence result for iterated contractions, which extends Theorem~6.5 of \cite{Pro10}.

\begin{corollary} 
\label{cor:third-convergence-theorem-CMS}
Let ${T \colon D \subset X \to X}$ be an operator of a cone metric space ${(X,d)}$ over a solid vector space
${(Y,\preceq)}$, and let ${E \colon D \to {\Rset}_+}$ be a function of initial conditions of $T$ with a gauge function $\varphi$ 
of order $r$ on an interval $J$.
If $T$ is an iterated contraction with respect to $E$ with control function $\beta$ satisfying \eqref{eq:standard-condition-beta-1} and \eqref{eq:standard-condition-beta-2}, then the following statements hold true:
\begin{enumerate}[(i)]
\item
\textsc{Convergence}. Starting from any initial point $x_0$ of $T$, the Picard iteration \eqref{eq:Picard-iteration} is well-defined, 
remains in the closed ball ${\overline{U}(x_0,\rho)}$ and converges to a point ${\xi \in \overline{U}(x_0,\rho)}$, where
\[
\rho = \frac{d(x_0,Tx_0)}{1 - \beta(E(x_0))} \, .
\]
\item
\textsc{A priori estimate}. For all ${n \ge 0}$ we have the following error estimate
\[
d(x_n, \xi) \preceq \frac{\theta^n \lambda^{S_n(r)}}{1 - \theta \, \lambda^{r^n}} \,  d(x_0,Tx_0),
\]
where
$\lambda = \phi(E(x_0))$, $\theta  = \psi(E(x_0))$, ${\phi \colon J \to [0,1]}$ is a nondecreasing function satisfying
\eqref{eq:standard-condition-phi} and $\psi \colon J \to {\Rset}_+$ is a function satisfying \eqref{eq:standard-condition-psi}.
\item
\textsc{First a posteriori estimate}. For all ${n \ge 0}$ we have the following error estimate
\[
d(x_n, \xi) \preceq \frac{d(x_n,x_{n+1})}{1 - \beta(E(x_n))} \, .
\]
\item
\textsc{Second a posteriori estimate}. For all ${n \ge 0}$ we have the following error estimate
\[
d(x_{n+1}, \xi) \preceq \frac{\theta_n \lambda_n}{1 - \theta_n (\lambda_n)^r} \, d(x_n,x_{n+1}).
\]
where ${\lambda_n = \phi(E(x_n))}$ and ${\theta_n = \psi(E(x_n))}$.
\item 
\textsc{Some other estimates}. For all ${n \ge 0}$ we have
\[
d(x_{n+1},x_{n+2}) \preceq \theta \, \lambda^{r^n} \, d(x_n,x_{n+1})
\,\, \text{ and } \,\, 
d(x_n,x_{n+1}) \preceq  \theta^n \, \lambda^{S_n(r)} \, d(x_0,x_1).
\]
\item \textsc{Existence of a fixed point}. 
If ${\overline{U}(x_0,\rho)} \subset D$ and $T$ is continuous, then $\xi$ is a fixed point of $T$.
\end{enumerate}
\end{corollary}

\begin{remark}
Each of the theorems and corollaries of this section is a generalization of the Banach Contraction Principle \cite{Ban22} 
(see also \cite{Pro13})
as well as of the Iterated contraction principle given in \cite[Chap.~12]{OR70} and \cite{Pro13}.
For example, Corollary~\ref{cor:third-convergence-theorem-CMS} with ${E(x) = d(x,Tx)}$, ${\varphi(t) \equiv \lambda \, t}$ and ${\beta(t) \equiv \lambda}$, where ${\lambda \in [0,1)}$, yields Theorem~10.1 of \cite{Pro13}.   
\end{remark}

%
%

\section{Some inequalities in $\Kset^n$}
\label{sec:Some-inequalities}

In this and the next sections, we use the following notations and conventions.
The vector space $\Kset^n$ is equipped with the $p$-norm
\begin{equation} \label{eq:p-norm}
\|x\|_p = \left( \sum _{i = 1} ^n |x_i|^p \right) ^{1/p}   \quad\text{for some } 1 \le p \le \infty.
\end{equation}
The Banach space ${(\Rset^n,\|\cdot\|_p)}$ is equipped with the coordinate-wise ordering $\preceq$ defined by
\begin{equation} \label{eq:coordinate-wise-ordering}
x \preceq y  \quad\text{if and only if}\quad x_i \le y_i \,\, \text{ for each } \,\, i \in I_n \, .
\end{equation}
Here and in what follows, we denote by $I_n$ the set of indices ${1, \ldots, n}$, i.e. ${I_n = \{1, \ldots, n\}}$.
It is easy to see that ${(\Rset^n,\|\cdot\|_p, \preceq)}$ is a solid vector space.
We define the map ${\|\cdot\| \colon \Kset^n \to \Rset^n }$ by
\begin{equation} \label{eq:cone-norm-in-K^n}
\|x\| = (|x_1|,\ldots,|x_n|)
\end{equation}
Then ${(\Kset^n,\|\cdot\|)}$ is a cone normed space over $\Rset^n$.

Furthermore, for two vectors ${x \in \Kset^n}$ and ${y \in \Rset^n}$ we denote by ${\displaystyle \frac{x}{y}}$ a vector in ${\Rset^n}$ defined by
\begin{equation} \label{eq:quotient-of-vectors}
\frac{x}{y} = \left( \frac{|x_1|}{y_1},\cdots,\frac{|x_n|}{y_n} \right)
\end{equation}
provided that $y$ has only nonzero components.

We use the function ${d \colon \Kset^n \to \Rset^n}$ defined by ${d(x) = (d_1(x),\ldots,d_n(x))}$ with
\begin{equation} \label{eq:function-d}
d_i(x) = \min_{j \ne i} |x_i - x_j| \qquad (i = 1,\ldots,n).
\end{equation}

\begin{proposition}[\cite{PC14a}] \label{prop:inequality-1}
Let ${u,v \in \Kset^n}$ and ${1 \le p \le \infty}$. If the vector $v$ has distinct components, then for all ${i,j \in I_n}$ the following two inequalities hold: 
\begin{enumerate}[(i)]
	\item 
$|u_i - u_j| \ge \left( 1 - 2^{1/q} \displaystyle\left\| \frac{u - v}{d(v)} \right\|_p \right) |v_i - v_j|$,
	\item
$|u_i - v_j| \ge \left( 1 - \displaystyle\left\| \frac{u - v}{d(v)} \right\|_p \right) |v_i - v_j|$.
\end{enumerate}
\end{proposition}

The next two lemmas are immediate consequences of the previous one.

\begin{proposition} \label{prop:inequality-2}
Let ${u,v \in {\Kset}^n}$ and ${1 \le p \le \infty}$.
If the vector $v$ has distinct components, then
\[
d(u) \succeq \left( 1 - 2^{1/q} \left\| \frac{u - v}{d(v)} \right\|_p \right) d(v).
\]
\end{proposition}

\begin{proposition} \label{prop:distinct-components}
Let $u,v \in {\Kset}^n$ and $1 \le p \le \infty$. If the vector $v$ has distinct components and
\[
\left\| \frac{u - v}{d(v)} \right\|_p  < \frac{1}{2^{1/q}} \, ,
\]
then the vector $u$ also has distinct components.
\end{proposition}

Given ${-\infty \le r \le \infty}$, we define the power mean function ${M_r \colon \Kset^n \to \Rset^n}$ by
\[
M_r(x) = \left( \frac{1}{n} \sum_{i=1}^n {|x_i|^r} \right)^{1 / r} .
\]
The value of $M_r(x)$ for ${r = 0,\pm\infty}$ is assumed to be the limit as ${r \to 0,\pm\infty}$.

\begin{proposition}[Power Mean Inequality] \label{prop:power-mean-inequality}
If ${-\infty \le r < s \le \infty}$, then 
\[
M_r(x) \le M_s(x) \quad\mbox{for every}\quad x \in \Kset^n,
\]
and the equality holds only if all the components of $x$ are equal to each other. 
\end{proposition}

The next proposition can easily be proved by the power mean inequalities $M_0 \le M_1$ and $M_1 \le M_p$ for $p \ge 1$.

\begin{proposition}[\cite{PC14a,PP13}] \label{prop:product-lemma}
Let $u \in {\Kset}^n$ and ${1 \le p \le \infty}$. Then
\[
\left|\prod_{i=1}^n (1 + u_i) \right| \le \left( 1 + \frac{\|u\|_p}{n^{1 / p}} \right)^n
\quad\text{and}\quad
\left|\prod_{i=1}^n (1 + u_i) - 1 \right| \le \left( 1 + \frac{\|u\|_p}{n^{1 / p}} \right)^n - 1 .
\]
\end{proposition}

Throughout the next sections, for a given ${1 \le p \le \infty}$, we always denote by $q$ the conjugate exponent of $p$, that is,
\begin{equation} \label{eq:conjugate-exponent}
1 \le q \le \infty \quad\text{and}\quad \frac{1}{p} + \frac{1}{q} = 1.
\end{equation}

%
%

\section{Local convergence of the first kind of the Weierstrass method}
\label{sec:Local-convergence-of-the-Weierstrass-method-I}

Let $f \in \Kset[z]$ be a polynomial of degree $n \ge 2$ which has $n$ simple zeros in $\Kset$ and $\xi$ be a root-vector of $f$.
In this section we study the convergence of the Weierstrass method \eqref{eq:Weierstrass-iteration} with respect to the function of initial conditions ${E \colon \Kset^n \to \Rset_+}$ defined by
\begin{equation} \label{eq:FIC1-SM}
E(x) = \left\| \frac{x - \xi}{d(\xi)} \right\|_p ,
\end{equation}
where $1 \le p \le \infty$. 
We prove that the Weierstrass iteration function $T$ defined by \eqref{Weierstrass-iteration-function} is an iterated contraction with respect to $E$ at the point $\xi$. 
As a result, we obtain a local convergence theorem with error estimates for the Weierstrass method, which improves the results of
Dochev \cite{Doc62b},
Kyurkchiev and Markov \cite{KM83}, 
Yakoubsohn \cite{Yak02} and
Proinov and Petkova \cite{PP13}.

\begin{lemma} \label{lem:FIC1-SM-properties}
Let ${f \in \Kset[z]}$ be a polynomial of degree ${n \ge 2}$ which splits in $\Kset$, ${\xi \in \Kset^n}$ be a root-vector of $f$, ${x \in \Kset^n}$ and ${1 \le p \le \infty}$. Then for ${i \ne j}$,
\[
|x_i - x_j| \ge (1 - 2^{1 / q} E(x)) \, d_j(\xi) \quad\text{and}\quad    |x_i - \xi_j| \ge (1 - E(x)) \, d_i(\xi),  
\]
where ${E \colon \Kset^n \to \Rset_+}$ is defined by \eqref{eq:FIC1-SM}. 
\end{lemma}

\begin{proof}
Setting in Proposition~\ref{prop:inequality-1} ${u = x}$ and ${v = \xi}$ and taking into account the definition of $d(\xi)$, we obtain the statement of the lemma.
\end{proof}

\begin{lemma} \label{lem:Weierstrass-identity}
Let $f \in \Kset[z]$ be a polynomial of degree $n \ge 2$ which splits in $\Kset$, $\xi$ be a root-vector of $f$, and let  
${x \in \Kset^n}$ be a vector with distinct components. Then for every ${i \in I_n}$,
\begin{equation} \label{eq:Weierstrass-identity}
T_i(x) - \xi_i = \left( \, \prod_{j \neq i} { \left( 1 + u_j \right)} - 1 \right) |x_i - \xi_i| 
\end{equation}
where 
\begin{equation} \label{eq:Weierstrass-u}
u_j = \frac{x_j - \xi_j}{x_i - x_j} \, . 
\end{equation}
\end{lemma}

\begin{proof}
Let ${i \in I_n}$ be fixed.
Taking into account the identity \eqref{eq:root-vector}, we obtain
\[
W_i(x) = (x_i - \xi_i) \prod_{j \neq i} {\frac{x_i - \xi_j}{x_i - x_j}} = (x_i - \xi_i) \prod_{j \neq i} { \left( 1 + u_j \right)}.
\]
From this and \eqref{Weierstrass-iteration-function}, we get \eqref{eq:Weierstrass-identity}.
\end{proof}

\begin{lemma} \label{lem:Weierstrass-1}
Let $f \in \Kset[z]$ be a polynomial of degree $n \ge 2$ which has $n$ simple zeros in $\Kset$, $\xi$ be a root-vector of $f$ 
and $1 \le p \le \infty$.
Suppose a vector $x \in \Kset^n$ satisfies
\begin{equation} \label{eq:Weierstrass-1-initial-condition}
E(x) = \left\| \frac{x - \xi}{d(\xi)} \right\|_p < \frac{1}{2^{1/q}} \, ,
\end{equation}
where the function $E$ is defined by \eqref{eq:FIC1-SM}. Then ${x \in \mathcal{D}}$,
\begin{equation} \label{eq:Weierstrass-1-iterated-contraction}
E(Tx) \le \varphi(E(x))
\quad\text{and}\quad
\|Tx - \xi\| \preceq \phi(E(x)) \, \|x - \xi\|,
\end{equation}
where $T$ is the Weierstrass iteration function defined by \eqref{Weierstrass-iteration-function}, and the real functions $\varphi$ and 
$\phi$ are defined by
\begin{equation} \label{eq:Weierstrass-1-phi}
\phi(t) = \left( 1 + \frac{t}{(n - 1)^{1/p} (1 - 2^{1/q} \, t)} \right)^{n - 1} - 1
\quad\text{and}\quad
\varphi(t) = t \, \phi(t).
\end{equation}
\end{lemma}

\begin{proof}
It follows from Proposition~\ref{prop:distinct-components} with ${u = x}$ and ${v = \xi}$ that ${x \in \mathcal{D}}$.
Let ${i \in I_n}$ be fixed.
Combining Lemma~\ref{lem:Weierstrass-identity} and Proposition~\ref{prop:product-lemma}, we obtain
\begin{equation} \label{eq:Weierstrass-1-iterated-contraction-A}
|T_i(x) - \xi_i| \le
\left[ \left( 1 + \frac{\|u\|_p}{(n-1)^{1 / p}} \right)^{n-1} - 1 \right] |x_i - \xi_i| ,
\end{equation}
where $u = (u_j)_{j \neq i} \in \Kset^{n - 1}$ and $u_j$ is defined by \eqref{eq:Weierstrass-u}.
It follows from Lemma~\ref{lem:FIC1-SM-properties} that
\begin{equation}
|u_j| = \left| \frac{x_j - \xi_j}{x_i - x_j} \right| \le \frac{|x_j - \xi_j|}{( 1 - 2^{1/q} E(x) ) \, d_j(x)}
\end{equation}
which yields
\[
\|u\|_p \le \frac{E(x)}{1 - 2^{1/q} E(x)} .
\]
From \eqref{eq:Weierstrass-1-iterated-contraction-A} and the last inequality, we get
\begin{equation} \label{eq:Weierstrass-1-iterated-contraction-B}
|T_i(x) - \xi_i| \le \phi(E(x)) |x_i - \xi_i|
\end{equation}
which yields the second inequality in \eqref{eq:Weierstrass-1-iterated-contraction}. The first inequality in 
\eqref{eq:Weierstrass-1-iterated-contraction} follows from \eqref{eq:Weierstrass-1-iterated-contraction-B} dividing by $d_i(\xi)$ and taking the $p$-norm.
\end{proof}

Now we are ready to state the main result of this section.

\begin{theorem} 
\label{thm:first-local-Convergence-theorem-Weierstrass}
Let ${f \in \Kset[z]}$ be a polynomial of degree $n \ge 2$ which has $n$ simple zeros in $\Kset$, $\xi$ be a root-vector of $f$ and
${1 \le p \le \infty}$.
Suppose $x^0 \in \Kset^n$ is an initial guess such that
\begin{equation} \label{eq:first-local-Convergence-theorem-Weierstrass-initial-condition}
E(x^0) = \left\| \frac{x^0 - \xi}{d(\xi)} \right\|_p <
R(n, p) =\frac{2^{1/(n-1)}-1}{2^{1 / q} \left( 2^{1/(n-1)}-1 \right) + (n-1)^{-1/p}} \, ,
\end{equation}
where the function $E$ is defined by \eqref{eq:FIC1-SM}.
Then the Weierstrass iteration \eqref{eq:Weierstrass-iteration} is well-defined and converges quadratically to $\xi$ with error estimates
\begin{equation} \label{eq:first-local-Convergence-theorem-Weierstrass-error-estimates}
	\|x^{k + 1} - \xi\| \preceq \lambda^{2^k} \, \|x^k - \xi\|
	\quad\text{and\quad}
\|x^k - \xi\| \preceq \lambda^{2^k - 1} \, \|x^0 - \xi\|
\end{equation}
for all $k \ge 0$, where ${\lambda = \phi(E(x^0))}$ and the real function $\phi$ is defined by \eqref{eq:Weierstrass-1-phi}.
\end{theorem}

\begin{proof}
By Example~\ref{exmp:quasi-homogeneous-function-1}, the function $\varphi$ defined by \eqref{eq:Weierstrass-1-phi} is quasi-homogeneous of the second degree on ${[0,1/2^{1/q})}$. On the other hand, ${R = R(n,p)}$ is a fixed point of $\varphi$ in ${(0,1/2^{1/q})}$.
Then according to Proposition~\ref{prop:gauge-function-sufficient-condition}, $\varphi$ is a strict gauge function of the second order on the interval ${J = [0, R)}$. 

Now we shall apply Corollary~\ref{thm:first-local-convergence-theorem-CMS-phi} to the Weierstrass iteration function 
$T \colon \mathcal{D} \subset \Kset^n \to \Kset^n$.
It follows from Lemma~\ref{lem:Weierstrass-1} that:

$\bullet$ ${E \colon \mathcal{D} \to \Rset_+}$ is a functions of initial conditions of $T$ with gauge function $\varphi$ 
of order ${r = 2}$ on $J$.

$\bullet$ ${T \colon \mathcal{D} \to \Kset^n}$ is an iterated contraction with respect to $E$ at $\xi$ with control function $\phi$. 

It remains to prove that every vector ${x^0 \in \mathcal{D}}$ with ${E(x^0) \in J}$ is an initial point of $T$. 
According to Proposition~\ref{prop:initial-point-test} it is sufficient to prove that 
\begin{equation} \label{eq:initial-point-test-2}
x \in \mathcal{D} \text{ with } E(x) \in J \text{ implies } Tx \in \mathcal{D}.	
\end{equation}
Let ${x \in \mathcal{D}}$ be such that ${E(x) \in J}$.
From ${x \in \mathcal{D}}$, we get ${Tx \in \Kset^n}$.
From the first inequality in \eqref{eq:Weierstrass-1-iterated-contraction} taking into account that ${E(x) \in J}$ and 
${\varphi \colon J \to J}$, we conclude that ${E(Tx) \in J}$.
Then applying Lemma~\ref{lem:Weierstrass-1} to $Tx$ instead of $x$, we deduce ${Tx \in \mathcal{D}}$ which proves 
\eqref{eq:initial-point-test-2}.
 
Now the statement of the theorem follows from Corollary~\ref{thm:first-local-convergence-theorem-CMS-phi}.
\end{proof}

Let $0 < h < 1$ be a given number. Solving the equation $\phi(t) = h$ in the
interval $(0,R(n, p))$ we can formulate Theorem~\ref{thm:first-local-Convergence-theorem-Weierstrass} in the following equivalent form.

\begin{theorem} 
\label{thm:first-local-Convergence-theorem-Weierstrass-variant}
Let $f \in \Kset[z]$ be a polynomial of degree $n \ge 2$ which has $n$ simple zeros in $\Kset$, $\xi$ a root-vector of $f$, 
${1 \le p \le \infty}$ and ${0 < h < 1}$.
Suppose $x^0 \in \Kset^n$ is an initial guess such that
\begin{equation} \label{thm:first-local-Convergence-theorem-Weierstrass-variant-initial-conditions}
\left\| \frac{x^0 - \xi}{d(\xi)} \right\|_p \le R(n,p,h) = 
\frac{(1 + h)^{1/(n-1)}-1}{2^{1 / q} \left( (1 + h)^{1/(n-1)}-1 \right) + (n-1)^{-1/p}} \, .
\end{equation}
Then the Weierstrass iteration \eqref{eq:Weierstrass-iteration} is well-defined and converges quadratically to $\xi$ with error estimates
\begin{equation} \label{eq:first-local-Convergence-theorem-Weierstrass-variant-error-estimates}
	\|x^{k + 1} - \xi\| \preceq h^{2^k} \, \|x^k - \xi\|
	\quad\text{and}\quad
\|x^k - \xi\| \preceq h^{2^k - 1} \, \|x^0 - \xi\|
\end{equation}
for all $k \ge 0$.
\end{theorem}
 
In the next two corollaries we denote by $\fsep$ the separation number of $f$ which is defined to be the minimum distance between two zeros of $f$, that is,
\begin{equation} \label{eq:separation-number}
\fsep = \min_{i \ne j} |\xi_i - \xi_j|.
\end{equation}

\begin{corollary} \label{cor:Dochev} 
Let $f \in \Kset[z]$ be a polynomial of degree $n \ge 2$ which has $n$ simple zeros in $\Kset$, $\xi$ be a root-vector of $f$, 
$1 \le p \le \infty$ and $0 < h < 1$.
Suppose $x^0 \in \Kset^n$ is an initial guess such that
\begin{equation} \label{eq:Dochev-initial-conditions}
\|x^0 - \xi\|_p \le \rho = R(n,p,h) \, \fsep,
\end{equation}
where $R(n,p,h)$ is defined in \eqref{thm:first-local-Convergence-theorem-Weierstrass-variant-initial-conditions}.
Then the Weierstrass iteration \eqref{eq:Weierstrass-iteration} is well-defined and converges quadratically to $\xi$ with error estimates 
\eqref{eq:first-local-Convergence-theorem-Weierstrass-variant-error-estimates} and
\begin{equation} \label{eq:Dochev-error-estimates3}
\|x^k - \xi\|_p \le \rho \, h^{2^k - 1}.
\end{equation}
\end{corollary}

\begin{proof}
It follows immediately from Theorem~\ref{thm:first-local-Convergence-theorem-Weierstrass-variant} and the obvious inequality 
\[
E(x) \le \frac{\|x - \xi\|_p}{\fsep}
\]
which holds for every ${x \in \Kset^n}$.
The estimate \eqref{eq:Dochev-error-estimates3} follows from the second estimate in 
\eqref{eq:first-local-Convergence-theorem-Weierstrass-variant-error-estimates} and \eqref{eq:Dochev-initial-conditions}.
\end{proof}

Corollary~\ref{cor:Dochev} is a generalization and improvement of the result of Dochev \cite{Doc62b} (see also \cite{ID63}). 
In the case $p = \infty$, he proved the estimate \eqref{eq:Dochev-error-estimates3}  under the initial condition 
\eqref{eq:Dochev-initial-conditions}.
 
\begin{corollary} \label{cor:Kyurkchiev-Markov}
Let $f \in \Kset[z]$ be a polynomial of degree $n \ge 2$ which has $n$ simple zeros in $\Kset$, $\xi$ be a root-vector of $f$, 
$1 \le p \le \infty$, $0 < h < 1$ and
${0 < c \le R(n, p) \, \fsep }$, where $R(n, p)$ is defined in \eqref{eq:first-local-Convergence-theorem-Weierstrass-initial-condition}.
Suppose $x^0 \in \Kset^n$ is an initial guess satisfying
\begin{equation} \label{eq:Kyurkchiev-Markov-initial-condition}
\left\| x^0 - \xi \right\|_p \le c \, h .
\end{equation}
Then the Weierstrass iteration \eqref{eq:Weierstrass-iteration} is well-defined and converges quadratically to $\xi$ with error estimates 
\eqref{eq:first-local-Convergence-theorem-Weierstrass-variant-error-estimates} and
\begin{equation} \label{eq:Kyurkchiev-Markov-estimate3}
\|x^k - \xi\|_p \le c \, h^{2^k}.
\end{equation}
\end{corollary}

\begin{proof}
Obviously, the initial guess $x^0$ satisfies the initial condition 
\eqref{eq:first-local-Convergence-theorem-Weierstrass-initial-condition} 
of Theorem~\ref{thm:first-local-Convergence-theorem-Weierstrass}. 
Therefore, we obtain the estimates \eqref{eq:first-local-Convergence-theorem-Weierstrass-error-estimates} which imply the estimates
\eqref{eq:first-local-Convergence-theorem-Weierstrass-variant-error-estimates}. 
Indeed, we have
\[
\lambda = \phi(E(x^0)) \le \phi(c \, h) \le h \, \phi(c) \le h \, \phi(R(n, p)) = h
\]
since $\phi$ is quasi-homogeneous of the first degree on ${[0, 1 / 2^{1/q})}$.
The estimate \eqref{eq:Kyurkchiev-Markov-estimate3} follows from
the second estimate in \eqref{eq:first-local-Convergence-theorem-Weierstrass-variant-error-estimates} and  
\eqref{eq:Kyurkchiev-Markov-initial-condition}.
\end{proof}

Corollary~\ref{cor:Kyurkchiev-Markov} improves the result of Kyurkchiev and Markov \cite{KM83}. 
In the case $p = \infty$, they have proved the estimate \eqref{eq:Kyurkchiev-Markov-estimate3} under the initial condition 
\eqref{eq:Dochev-initial-conditions} but with a stronger condition for $c$.

\medskip
In 2002, Yakoubsohn \cite{Yak02} published a $\gamma$-theorem for Weierstrass method.
He introduce the quantity
\begin{equation}
	\gamma(f) = \max_{1 \le i \le n} {\gamma(f,\xi_i)}
	\quad\text{where}\quad \gamma(f,x) = \max_{k>1}{\left| \frac{f^{(k)}(x)}{k! f'(x)} \right|^{1 / (k-1)}}.
\end{equation}
Recall that $\gamma(f,x)$ has been introduced by Smale in his famous work \cite{Sma86}.

\begin{corollary} \label{cor:Yakoubsohn-2002}
Let $f \in \Kset[z]$ be a polynomial of degree $n \ge 2$ which has $n$ simple zeros in $\Kset$, $\xi$ be a root-vector of $f$ and 
${1 \le p \le \infty}$.
If ${x^0 \in \Kset^n}$ is an initial guess satisfying
\begin{equation} \label{eq:Yakoubsohn-initial-condition}
\left\| \frac{x^0 - \xi}{d(\xi)} \right\|_p \le \frac{R(n,p,h)}{2 \, \gamma(f)} \, ,
\end{equation}
where $R(n, p,h)$ is defined in \eqref{thm:first-local-Convergence-theorem-Weierstrass-variant-initial-conditions}, 
then the Weierstrass iteration \eqref{eq:Weierstrass-iteration} is well-defined and converges quadratically to $\xi$ with error estimates 
\eqref{eq:first-local-Convergence-theorem-Weierstrass-variant-error-estimates}. 
\end{corollary}

\begin{proof}
It follows from Theorem~\ref{thm:first-local-Convergence-theorem-Weierstrass-variant} and the inequality
${\gamma(f) \ge 1 / (2 \, \fsep)}$ which is due to Yakoubsohn \cite{Yak00}.
\end{proof}

Corollary~\ref{cor:Yakoubsohn-2002} generalizes and improves the result of Yakoubsohn \cite{Yak02}. 
In the case ${p = \infty}$ and ${h = 1/2}$, he proved the second estimate estimate in 
\eqref{eq:first-local-Convergence-theorem-Weierstrass-variant-error-estimates} 
under a stronger initial condition than \eqref{eq:Yakoubsohn-initial-condition}.

%
%

\section{Local convergence of the second kind of the Weierstrass method}
\label{sec:Local-convergence-of-the-Weierstrass-method-II}

Let ${f \in \Kset[z]}$ be a polynomial of degree ${n \ge 2}$ which splits in $\Kset$, and let ${\xi \in \Kset^n}$ be a root-vector of $f$.
In this section we study the convergence of the Weierstrass method \eqref{eq:Weierstrass-iteration} with respect to the function of 
initial conditions ${E \colon \mathcal{D} \to \Rset_+}$ defined by
\begin{equation} \label{eq:FIC2-SM}
E(x) = \left\| \frac{x - \xi}{d(x)} \right\|_p ,
\end{equation}
where $1 \le p \le \infty$.
Recall that $\mathcal{D}$ denotes the set of all vectors in $\Kset^n$ with distinct components.
We prove that the Weierstrass iteration function $T$ defined by \eqref{Weierstrass-iteration-function} 
is also an iterated contraction with respect to this function $E$ at the point $\xi$. 
The main result of this section generalizes, improves and complements the results of 
Wang and Zhao \cite{WZ91}, 
Tilli \cite{Til98} and 
Han \cite{Han00}.

\begin{lemma} \label{lem:FIC2-SM-properties}
Let ${f \in \Kset[z]}$ be a polynomial of degree ${n \ge 2}$ which splits in $\Kset$, ${\xi \in \Kset^n}$ be a root-vector of $f$, 
${x \in \Kset^n}$ and ${1 \le p \le \infty}$. Then for ${i \ne j}$,
\[
|x_i - \xi_j| \ge (1 - E(x)) \, d_i(x) \quad\text{and}\quad |x_i - x_j| \ge d_j(x),
\]
where ${E \colon \mathcal{D} \to \Rset_+}$ is defined by \eqref{eq:FIC2-SM}.
\end{lemma}

\begin{proof}
Setting in Proposition~\ref{prop:inequality-1}(ii) ${u = \xi}$ and ${v = x}$ and taking into account the definition of $d(x)$, we obtain the first conclusion of the lemma. The second conclusion is obvious.
\end{proof}

\begin{lemma} \label{lem:Weierstrass-2}
Let $f \in \Kset[z]$ be a polynomial of degree $n \ge 2$ which splits in $\Kset$, ${\xi \in\Kset^n}$ be a root-vector of $f$ 
and $1 \le p \le \infty$.
Suppose ${x \in\Kset^n}$ is a vector with distinct components such that
\begin{equation} \label{eq:Weierstrass-2-initial-condition}
\psi(E(x)) > 0,
\end{equation}
where the function $E \colon \mathcal{D} \to \Rset_+$ is defined by \eqref{eq:FIC2-SM} and the real function $\psi$ is defined by
\begin{equation} \label{eq:Weierstrass-2-psi}
\psi(t) = 1 - 2^{1/q} \, t \left( 1 + \frac{t}{(n - 1)^{1 / p}} \right)^{n - 1} .
\end{equation}
Then $f$ has only simple zeros in $\Kset$, ${Tx}$ has pairwise distinct components,
\begin{equation} \label{eq:Weierstrass-2-iterated-contraction}
E(Tx) \le \varphi(E(x))
\quad\text{and}\quad
\|Tx - \xi\| \preceq \beta(E(x)) \|x - \xi\|,
\end{equation}
where $T \colon \mathcal{D} \subset \Kset^n \to \Kset^n$ is the Weierstrass iteration function defined by 
\eqref{Weierstrass-iteration-function}, and the real functions $\varphi$ and $\beta$ are defined by
${\varphi(t) = t \, \beta(t) / \psi(t)}$ and
\begin{equation} \label{eq:Weierstrass-2-beta}
	\beta(t) = \left( 1 + \frac{t}{(n - 1)^{1 / p}} \right)^{n - 1} - 1
\end{equation}
\end{lemma}

\begin{proof}
Note that condition \eqref{eq:Weierstrass-2-initial-condition} implies that $E(x) < 1/2^{1/q}$ 
since the function $\psi$ is decreasing on $\Rset_+$ and ${\psi(1/2^{1/q}) < 0}$.
Now it follow from Lemma~\ref{prop:distinct-components} that the vector $\xi$ has distinct components, which means that $f$ has only simple zeros in $\Kset$.
The second inequality in \eqref{eq:Weierstrass-2-iterated-contraction} follows from \eqref{eq:Weierstrass-1-iterated-contraction-A} taking into account that in this case $\|u\|_p \le E(x)$.
It remains to prove ${Tx \in \mathcal{D}}$ and the first inequality in \eqref{eq:Weierstrass-2-iterated-contraction}.
Applying Proposition~\ref{prop:inequality-2} with $u = Tx$ and $v = x$, we get
\begin{equation} \label{eq:lower-estimate-d(tx)-A}
d(Tx) \succeq \left( 1 - 2^{1/q} \left\| \frac{Tx - x}{d(x)} \right\|_p \right) d(x).
\end{equation}
From the triangle inequality and the second inequality in \eqref{eq:Weierstrass-2-iterated-contraction}, we obtain
\[
\|Tx - x\| \preceq \|Tx - \xi\| + \|x - \xi\| \preceq (1 + \beta(E(x)) \|x - \xi\|
\]
which yields
\[
\left\| \frac{Tx - x}{d(x)} \right\| \preceq (1 +  \beta(E(x)) \left\| \frac{x - \xi}{d(x)} \right\| .
\]
Taking the $p$-norm, we get
\[
\left\| \frac{Tx - x}{d(x)} \right\|_p \le E(x) (1 + \beta(E(x)) .
\]
Now it follows from \eqref{eq:lower-estimate-d(tx)-A} that
\begin{equation} \label{eq:lower-estimate-d(tx)}
d(Tx) \succeq \psi(E(x)) \, d(x).
\end{equation}
It follows from this and \eqref{eq:Weierstrass-2-initial-condition} that ${Tx \in \mathcal{D}}$.
Combining \eqref{eq:lower-estimate-d(tx)} with the second inequality in \eqref{eq:Weierstrass-2-iterated-contraction} 
and taking into account 
\eqref{eq:Weierstrass-2-initial-condition}, we get
\begin{equation}  \label{eq:Weierstrass-2-FIC-A}
\left\| \frac{Tx - \xi}{d(Tx)} \right\| \preceq \phi(E(x)) \left\| \frac{x - \xi}{d(x)} \right\| ,
\end{equation}
where the real function $\phi$ is defined by ${\phi(t) = \beta(t) / \psi(t)}$.
Taking the $p$-norm in \eqref{eq:Weierstrass-2-FIC-A}, we obtain the first inequality in \eqref{eq:Weierstrass-2-iterated-contraction}.
\end{proof}

The next theorem is the first main result in this section. 

\begin{theorem} 
\label{thm:second-local-Convergence-theorem-Weierstrass}
Let $f \in \Kset[z]$ be a polynomial of degree $n \ge 2$ which splits in $\Kset$, ${\xi \in\Kset^n}$ be a root-vector of $f$ 
and $1 \le p \le \infty$.
Suppose ${x^0 \in \Kset^n}$ is an initial guess with distinct components satisfying
\begin{equation} \label{eq:second-local-Convergence-theorem-Weierstrass-initial-conditions}
h(E(x^0)) \le 2,
\end{equation}
where the function $E$ is defined by \eqref{eq:FIC2-SM} and the real function $h$ is defined by
\begin{equation} \label{eq:second-local-Convergence-theorem-Weierstrass-h}
h(t) = \left( 1 + 2 ^{1 / q} \, t \right)
\left( 1 + \frac{t}{(n - 1) ^{1 / p}} \right)^{n - 1} .
\end{equation}
Then $f$ has only simple zeros in $\Kset$ and the Weierstrass iteration \eqref{eq:Weierstrass-iteration} is well-defined and converges
to $\xi$ with error estimates
\begin{equation} \label{eq:second-local-Convergence-theorem-Weierstrass-error-estimates}
\|x^{k + 1} - \xi\| \preceq \theta \lambda^{2^k} \|x^k - \xi\|
\quad\text{and}\quad
\|x^k - \xi\| \preceq \theta^k \lambda^{2^k - 1} \|x^0 - \xi\|
\end{equation}
for all $k \ge 0$, where $\lambda = \phi(E(z^0))$, $\theta = \psi(E(z^0))$, the real function $\phi$ is defined by 
$\phi(t) = \beta(t) / \psi(t)$, and $\psi$ and $\beta$ are defined by \eqref{eq:Weierstrass-2-psi} and 
\eqref{eq:Weierstrass-2-beta}, respectively.
 
Moreover, if the inequality in \eqref{eq:second-local-Convergence-theorem-Weierstrass-initial-conditions} is strict, 
then the Weierstrass iteration converges quadratically to $\xi$. 
\end{theorem}

\begin{proof}
The function $h$ is increasing and continuous on $\Rset_+$ with
${h(0) = 1}$ and ${h(1 / 2^{1/q}) > 2}$. Therefore, there exists a unique solution
${R = R(n, p)}$  of the equation ${h(t) = 2}$ in the interval ${(0, 1 / 2^{1/q)}}$. 
Hence, the initial condition \eqref{eq:second-local-Convergence-theorem-Weierstrass-initial-conditions} is equivalent to $E(x^0) \in J$, where $J = [0,R]$.
It follows from ${h(R) = 2}$ that
\begin{equation} \label{eq:beta=psi}
\beta(R) = \psi(R) = \frac{1 - b R}{1 + b R} \, . 
\end{equation}
The function $\psi$ is decreasing on $J$ with
$\psi(0) = 1$ and $\psi(R) > 0$. Hence, 
\begin{equation} \label{eq:range-psi}
0 < \psi(t) \le 1
\quad\text{far all } \, t \in J
\end{equation}
The function $\beta$ is increasing on $J$ with $\beta(0) = 1$ and $\beta(R) < 1$. From Example~\ref{exmp:quasi-homogeneous-function-1}, 
we conclude that ${t \beta(t)}$ is a strict gauge function of the second order on $J$. It is easy to see that $\varphi$ is a quasi-homogeneous gauge function of the second degree on $J$. On the other hand, \eqref{eq:beta=psi} implies that $R$ is a fixed point of 
$\varphi$.
Then according to Proposition~\ref{prop:gauge-function-sufficient-condition}, $\varphi$ is a gauge function of order ${r = 2}$ on $J$. 

It follows from \eqref{eq:range-psi} that for every $x \in \mathcal{D}$ with $E(x) \in J$ 
condition \eqref{eq:Weierstrass-2-initial-condition} holds. 
Hence, Lemma~\ref{lem:Weierstrass-2} and Proposition~\ref{prop:initial-point-test} show that:

$\bullet$ ${E \colon \mathcal{D} \to \Rset_+}$ is a functions of initial conditions of $T$ with gauge function $\varphi$ 
of order ${r = 2}$ on $J$.

$\bullet$ ${T \colon \mathcal{D} \to \Kset^n}$ is an iterated contraction with respect to $E$ at the point $\xi$ 
with control function $\beta$. 

$\bullet$ Every vector ${x^0 \in \mathcal{D}}$ satisfying condition \eqref{eq:second-local-Convergence-theorem-Weierstrass-initial-conditions} is an initial point of $T$.

Now the statement of the theorem follows from Theorem~\ref{thm:first-local-convergence-theorem-CMS}.
\end{proof}

In the next lemma we give a lower bound for the quantity $R = R(n, p)$ defined in the proof of 
Theorem~\ref{thm:first-local-Convergence-theorem-Weierstrass}. 

\begin{lemma} \label{lem:second-local-Convergence-theorem-Weierstrass-R-lower-bound}
Let $n \ge 2$ and $1 \le p \le \infty$. Denote by $R = R(n, p)$ the unique positive solution of the equation $h(t) = 2$, where the real function $h$ is defined by \eqref{eq:second-local-Convergence-theorem-Weierstrass-h}. Then
\begin{equation} \label{eq:second-local-Convergence-theorem-Weierstrass-R-lower-bound}
R \ge \frac{n \, (2^{1 / n} - 1)}{(n - 1)^{1 / q} + 2^{1 / q}} \, .
\end{equation}
This inequality becomes an equality if and only if $n = 2$ and $p = 1$.
\end{lemma}

\begin{proof}
Let $a = (n - 1)^{1 / q}$, $b = 2^{1 / q}$ and $c = (n - 1)^{1 / p}$. Note that $a$, $b$ and $c$ are greater than 1 and $a \, c = n - 1$.
By the definition of $R$, we get
\begin{equation} \label{eq:h(R)}
(1 + b \, R) \left( 1 + R / c \right)^{n - 1} = 2.
\end{equation}
Using Bernoulli's inequality ${(1 + t)^n  \ge 1 + n t}$, we obtain
\[
\left( \frac{1 + (a + b) R / n}{1 + R / c} \right)^n = 
\left( 1 +\frac{(a + b) R / n - R / c}{1 + R / c} \right)^n \ge \frac{1 + b R}{1 + R / c}. 
\]
From this and \eqref{eq:h(R)}, we get
\begin{equation} \label{eq:Bernoulli}
 \left( 1 + (a + b) R / n \right)^n \ge 2
\end{equation}
which proves \eqref{eq:second-local-Convergence-theorem-Weierstrass-R-lower-bound}. 
The equality in \eqref{eq:Bernoulli}  holds if and only if ${b \, c = 1}$
which is equivalent to ${n = 2}$ and ${p = 1}$.
\end{proof}

Theorem~\ref{thm:first-local-Convergence-theorem-Weierstrass} together with 
Lemma~\ref{lem:second-local-Convergence-theorem-Weierstrass-R-lower-bound} immediately implies the following result, which
improves and complements the result of Han \cite{Han00} as well some other results.

\begin{corollary} \label{cor:second-local-Convergence-theorem-Weierstrass-Han}
Let $f \in \Kset[z]$ be a polynomial of degree $n \ge 2$ which splits in $\Kset$, ${\xi \in\Kset^n}$ be a root-vector of $f$ and 
$1 \le p \le \infty$.
If $x^0 \in \Kset^n$ is an initial guess with distinct components satisfying
\begin{equation} \label{eq:initial-condition-Han}
	E(x^0) = \left\| \frac{x^0 - \xi}{d(x^0)} \right\|_p \le 	\frac{n \, (2^{1 / n} - 1)}{(n - 1)^{1 / q} + 2^{1 / q}} \, ,
\end{equation}
then $f$ has only simple zeros in $\Kset$ and the Weierstrass iteration \eqref{eq:Weierstrass-iteration} is well-defined and converges
to $\xi$ with error estimates \eqref{eq:second-local-Convergence-theorem-Weierstrass-error-estimates}.
Moreover, the convergence is quadratic provided that ${n \ge 3}$ or ${p > 1}$ or \eqref{eq:initial-condition-Han} holds with strict inequality.
\end{corollary}


\begin{corollary} \label{cor:second-local-Convergence-theorem-Weierstrass-Wang-Zhao}
Let ${f \in \Kset[z]}$ be a polynomial of degree $n \ge 2$ which splits in $\Kset$, ${\xi \in\Kset^n}$ be a root-vector of $f$ 
and $1 \le p \le \infty$.
If ${x^0 \in \Kset^n}$ is a vector with distinct components satisfying
\begin{equation} \label{eq:initial-condition-Wang_Zhao}
	E(x^0) = \left\| \frac{x^0 - \xi}{d(x^0)} \right\|_p \le R =
	\frac{2^{1/(n - 1)}-1}{2^{1 + 1/q} \left( 2^{1/(n - 1)}-1 \right) + (n-1)^{-1/p}} \, ,
\end{equation}
then $f$ has only simple zeros in $\Kset$ and the Weierstrass iteration \eqref{eq:Weierstrass-iteration} is well-defined and converges quadratically to $\xi$ 
with error estimates \eqref{eq:second-local-Convergence-theorem-Weierstrass-error-estimates}.
\end{corollary}

\begin{proof}
It follows from Corollary~\ref{cor:second-local-Convergence-theorem-Weierstrass-Han} because the right side of 
\eqref{eq:initial-condition-Han} is less than the right side of \eqref{eq:initial-condition-Wang_Zhao}
\end{proof}

Corollary \eqref{cor:second-local-Convergence-theorem-Weierstrass-Wang-Zhao} is a generalization and improvement 
of the result of Wang and Zhao \cite{WZ91}. 
They have proved the convergence of the Weierstrass iteration \eqref{eq:Weierstrass-iteration} for a polynomial
${f \in \Cset[z]}$ under the stronger initial condition
\begin{equation}
\|x^0 - \xi \|_\infty < \frac{\sqrt[n - 1]{2} - 1}{4 \sqrt[n - 1]{2}-3} \, \min_{1 \le i \le n} d_i(x^0) .
\end{equation}
than \eqref{eq:initial-condition-Wang_Zhao} in the case ${p = \infty}$ .

\medskip
The next theorem is the second main result in this section. 

\begin{theorem} 
\label{thm:third-local-convergence-theorem-Weierstrass}
Let $f \in \Kset[z]$ be a polynomial of degree $n \ge 2$ which splits in $\Kset$, ${\xi \in\Kset^n}$ be a root-vector of $f$ 
and $1 \le p \le \infty$.
Suppose $c \colon [0,1) \to \Rset_+$ is a nondecreasing function such that
\begin{equation} \label{eq:condition-c}
	c(t) \le (n - 1)^{1 / p} \, ( (1 + t)^{1 / (n - 1)} - 1 )
	\quad\text{for all}\quad t \in [0,1).
\end{equation}
Suppose also there exists ${\sigma \in (0,1)}$ such that
\begin{equation} \label{eq:condition-sigma}
	t \, c(t) \le c(t^2) \, ( 1 - 2^{1 / q} \, (1 + t) \, c(t)
	\quad\text{for all}\quad t \in [0,\sigma],
\end{equation}
where $1 \le q \le \infty$ and $1/p + 1/q = 1$.
Let $x^0 \in \Kset^n$ be a vector with distinct components satisfying
\begin{equation} \label{eq:third-local-convergence-theorem-Weierstrass-initial-condition}
\left\| \frac{x^0 - \xi}{d(x^0)} \right\|_p \le c(\sigma).
\end{equation}
Then $f$ has only simple zeros in $\Kset$ and the Weierstrass iteration \eqref{eq:Weierstrass-iteration} is well-defined and converges quadratically to $\xi$ with error estimates
\begin{equation} 
\|x^{k+1}  - \xi\|  \preceq  \sigma^{2^k} \|x^k  - \xi\|
\quad\text{and}\quad
\|x^k  - \xi\|  \preceq  \sigma^{2^k - 1} \|x^0  - \xi\|
\end{equation}
for all $k \ge 0$.
\end{theorem}

\begin{proof}
For the sake of simplicity, let $b = 2^{1 / q}$.
It is easy to show that ${1 - b \, (1 + t) \, c(t) > 0}$ for all ${t \in [0,1)}$.
Define the functions $\psi$ and $\beta$ by \eqref{eq:Weierstrass-2-psi} and 
\eqref{eq:Weierstrass-2-beta} on the interval $J = [0, c(\sigma)]$.
It is easy to show that the function $\psi$ can also be rewritten in the form
\begin{equation} \label{eq:Weierstrass-2-psi-equivalent}
\psi(t) = 1 - 2^{1 / q} \, t \, (1 + \beta(t)).
\end{equation}
Condition \eqref{eq:condition-c} can be rewritten in the following equivalent form
\begin{equation}  \label{eq:beta-c-equivalent}
\beta(c(t)) \le t \quad\text{for all } \, t \in [0,1).
\end{equation}
Note that $\beta$ is an increasing function on $J$ with values in ${[0,1)}$ since
\[
{\beta(t) \le \beta(c(\sigma)) \le \sigma <1} \quad\text{for all}\quad t \in J.
\] 
The function $\psi$ is a decreasing function on $J$ with values in ${(0,1]}$ since
\[
\psi(t) \ge \psi(c(\sigma)) = 1 - b \, c(\sigma) (1 + \beta(c(\sigma))) \ge 1 - b \, (1 + \sigma) \, c(\sigma) > 0
\quad\text{for all } \, t \in J. 
\]
Therefore, we can define the function $\phi$ on $J$ by $\varphi(t) = t \,\beta(t) / \psi(t)$.
Condition \eqref{eq:condition-sigma} implies the following one
\begin{equation}  \label{eq:third-local-convergence-theorem-Weierstrass-c2-A}
\varphi(c(t)) \le c(t^2) \quad\text{for all } \, t \in [0,\sigma].
\end{equation}
Indeed, it follows from \eqref{eq:Weierstrass-2-psi-equivalent}, \eqref{eq:beta-c-equivalent}  and \eqref{eq:condition-sigma} that
\[
\varphi(c(t) = \frac{c(t) \, \beta(c(t)}{\psi(c(t))}
\le \frac{t \, c(t)}{1 - b \, (1 + t) \, c(t)} \le c(t^2).
\]
Note also that the function $\varphi$ is nondecreasing on $J$ and $\varphi(J) \subset J$ since
\[
\varphi(t) \le \varphi(c(\sigma)) = c(\sigma^r) \le c(\sigma)
\quad\text{for all}\quad t \in J.
\]
Consider again the Weierstrass iteration function ${T \colon \mathcal{D} \to \Kset^n}$ and the function 
${E \colon \mathcal{D} \to \Rset_+}$ defined by \eqref{eq:FIC2-SM}.  
For every $x \in D$ with $E(x) \in J$ condition \eqref{eq:Weierstrass-2-initial-condition} holds. Indeed, we have
\[
\psi(E(x)) \ge \psi(c(\sigma)) > 0.
\]
It follows from Lemma~\ref{lem:Weierstrass-2} that $E$ is a functions of initial conditions of $T$ with gauge function $\varphi$ on $J$,
and that $T$ is an iterated contraction with respect to $E$ at the point $\xi$ with control function $\beta$.
Now the statements of the theorem follow from Theorem~\ref{thm:third-local-convergence-theorem-CMS}.
\end{proof}

There are a lot of functions that satisfy condition \eqref{eq:condition-c} of 
Theorem~\ref{thm:third-local-convergence-theorem-Weierstrass}. 
For example, each of the functions
\begin{equation} \label{eq:special-functions-c}
	c(t) = \frac{\, 2 \, t - t^2}{2 \, (n - 1)^{1 / q}}
\qquad\text{and}\qquad
	c(t) = \frac{2 \, t}{(n - 1)^{1 / q} \, (t  + 2)} \, .
\end{equation}
satisfies \eqref{eq:condition-c}. This statement follows from the obvious inequalities
\[
(1 + t)^{1 / (n - 1)} - 1 \ge \frac{\ln(1 + t)}{n - 1}
\quad\text{and}\quad
\ln(1 + t) \ge \frac{2 t}{t + 2} \ge t - \frac{\, t^2}{2} .
\]

\medskip
Applying Theorem~\ref{thm:third-local-convergence-theorem-Weierstrass} with the first function $c \colon [0, 1) \to \Rset_+$ defined by \eqref{eq:special-functions-c} we obtain at the following result.

\begin{corollary} \label{cor:third-local-convergence-theorem-Weierstrass-corollary-1}
Let ${f \in \Kset[z]}$ be a polynomial of degree ${n > 2^{q + 1} + 1}$ which splits over $\Kset$, 
$\xi$ be a root-vector of $f$, ${1 < p \le \infty}$, and let ${\sigma \in (0,1/2]}$ be such that
\begin{equation}\label{eq:condition-sigma-corollary}
	\frac{(\sigma + 1) (2 - \sigma) (2 - \sigma^2)}{1 - \sigma} \le 2 \left( \frac{n - 1}{2} \right)^{1 / q}.
\end{equation}
Suppose $x^0 \in \Kset^n$ is a vector with distinct components satisfying
\begin{equation} \label{eq:third-local-convergence-theorem-Weierstrass-corollary-1-initial-condition}
\left\| \frac{x^0 - \xi}{d(x^0)} \right\|_p \le
\frac{\, 2 \, \sigma - \sigma^2}{2 \, (n - 1)^{1 / q}} \, .
\end{equation}
Then $f$ has only simple zeros in $\Kset$ and the Weierstrass iteration \eqref{eq:Weierstrass-iteration} is well-defined and converges quadratically to the vector $\xi$ with error estimates
\begin{equation}  \label{eq:third-local-convergence-theorem-Weierstrass-corollary-1-error-estimates}
\|x^{k+1}  - \xi\|  \preceq  \sigma^{2^k} \|x^k  - \xi\|
\quad\text{and}\quad
\|x^k  - \xi\|  \preceq  \sigma^{2^k - 1} \|x^0  - \xi\|,
\end{equation}
\begin{equation}  \label{eq:third-local-convergence-theorem-Weierstrass-corollary-1-error-estimates-Tilli}
\|x^k  - \xi\|_\infty  \preceq  \sigma^{2^k} \max_{1 \le i \le n} {d_i(x^0)}
\end{equation}
for all $k \ge 0$.
Besides, if $n \ge 2^{2 q + 1} + 1$, then condition \eqref{eq:condition-sigma-corollary} can be dropped.
\end{corollary}

\begin{proof}
The statement follows from Theorem~\ref{thm:third-local-convergence-theorem-Weierstrass} with the second function 
${c \colon [0, 1) \to \Rset_+}$ defined by \eqref{eq:special-functions-c}.
The  estimate \eqref{eq:third-local-convergence-theorem-Weierstrass-corollary-1-error-estimates-Tilli}
follows from the second estimate in \eqref{eq:third-local-convergence-theorem-Weierstrass-corollary-1-error-estimates} and the initial condition \eqref{eq:third-local-convergence-theorem-Weierstrass-corollary-1-initial-condition}.
It should be noted only that condition \eqref{eq:condition-sigma-corollary} 
is equivalent to \eqref{eq:condition-sigma} and that 
\eqref{eq:condition-sigma-corollary} is satisfied automatically if ${n \ge 2^{2 q + 1} + 1}$.
\end{proof}

Corollary \eqref{cor:third-local-convergence-theorem-Weierstrass-corollary-1} is a generalization and improvement of the result of 
Tilli \cite{Til98} who has proved the error estimate
\eqref{eq:third-local-convergence-theorem-Weierstrass-corollary-1-error-estimates-Tilli}  
under the condition \eqref{eq:third-local-convergence-theorem-Weierstrass-corollary-1-initial-condition} with ${p = \infty}$ 
for a polynomial ${f \in \Cset[z]}$ of degree ${n \ge 9}$.

\medskip
Applying Theorem~\ref{thm:third-local-convergence-theorem-Weierstrass} with the second function $c \colon [0, 1) \to \Rset_+$ defined by \eqref{eq:special-functions-c} we arrive at the following result.

\begin{corollary} \label{cor:third-local-convergence-theorem-Weierstrass-corollary-2}
Let ${f \in \Kset[z]}$ be a polynomial of degree ${n > 2^{q + 1} + 1}$ which splits over $\Kset$, $\xi$ be a root-vector of $f$ and 
${1 < p \le \infty}$.
Suppose $\sigma$ is a real number satisfying
\[
0 < \sigma \le \frac{(n - 1)^{1 / q} - 2^{1 + 1 / q}}{(n - 1)^{1 / q} + 2^{1 + 1 / q}}
\]
and $x^0 \in \Kset^n$ is  a vector with distinct components such that
\begin{equation} \label{eq:third-local-convergence-theorem-Weierstrass-corollary-2-initial-condition}
\left\| \frac{x^0 - \xi}{d(x^0)} \right\|_p \le
\frac{2 \, \sigma}{(n - 1)^{1 / q} \, (\sigma  + 2)} \, .
\end{equation}
Then $f$ has only simple zeros in $\Kset$ and the Weierstrass iteration \eqref{eq:Weierstrass-iteration} is well-defined and converges quadratically to $\xi$ with error estimates \eqref{eq:third-local-convergence-theorem-Weierstrass-corollary-1-error-estimates}.
\end{corollary}

Corollary~\ref{cor:third-local-convergence-theorem-Weierstrass-corollary-1} can be state in the following equivalent form.

\begin{corollary} \label{cor:third-local-convergence-theorem-Weierstrass-corollary-3}
Let ${f \in \Kset[z]}$ be a polynomial of degree ${n > 2^{q + 1} + 1}$ which splits over $\Kset$, $\xi$ be a root-vector of $f$ and 
${1 < p \le \infty}$.
Suppose $x^0 \in \Kset^n$ is a vector with distinct components satisfying
\begin{equation} \label{eq:third-local-convergence-theorem-Weierstrass-Tilli-1998-equivalent-initial-conditions}
\left\| \frac{x^0 - \xi}{d(x^0)} \right\|_p \le
\frac{2 \, [(n - 1)^{1 / q} - 2^{1 + 1 / q}]}{(n - 1)^{1 / q} \, [ 3 (n - 1)^{1 / q} + 2^{1 + 1 / q}]} \, .
\end{equation}
Then $f$ has only simple zeros in $\Kset$ and the Weierstrass iteration \eqref{eq:Weierstrass-iteration} is well-defined and converges quadratically to $\xi$ with error estimates
\begin{equation}
\|x^{k+1}  - \xi\|  \preceq  \lambda^{2^k} \|x^k  - \xi\|
\quad\text{and}\quad
\|x^k  - \xi\|  \preceq  \lambda^{2^k - 1} \|x^0  - \xi\|
\end{equation}
for all $k \ge 0$, where
\[
\lambda = \frac{2 (n - 1)^{1 / q} \, E(x^0)}{2 - (n - 1)^{1 / q} \, E(x^0)} \, .
\]
\end{corollary}

%
%

\section{Semilocal convergence of the Weierstrass method}
\label{sec:Semilocal-convergence-of-the-Weierstrass-method}

In this section, we prove a new convergence theorem for the Weierstrass method under computationally verifiable initial conditions.
The main result of this section generalizes, improves and complements all previous results in this area, which are due to 
Pre{\v s}i\'c \cite{Pre80},
Zheng \cite{Zhe82,Zhe87},
Wang and Zhao \cite{ZW93,WZ95},
Petkovi\'c, Carstensen and Trajkovi\'c \cite{PCT95},
Petkovi\'c \cite{Pet96},
Petkovi\'c, Herceg and Ili\'c \cite{PHI98},
Batra \cite{Bat98},
Han \cite{Han00},
Petkovi\'c and Herceg \cite{PH01} and
Proinov \cite{Pro06b}.
The new result in this section also gives computationally verifiable error estimates,
a localization formula for the Weierstrass iteration \eqref{eq:Weierstrass-iteration} as well as 
a sufficient condition for a polynomial ${f \in \Kset[z]}$ of degree ${n \ge 2}$ to have $n$ simple zeros in the field $\Kset$. 
Finally, we provide an example which shows the exactness of the error estimates of our semilocal theorem for the Weierstrass iterative method.

We study the convergence of the Weierstrass method \eqref{eq:Weierstrass-iteration} for a polynomial $f \in \Kset[z]$ 
with respect to the function of initial conditions
${E \colon \mathcal{D} \to \Rset_+}$ defined by
\begin{equation} \label{eq:FIC3-SM}
E(x) = \left \| \frac{W(x)}{d(x)} \right \|_p \qquad (1 \le p \le \infty).
\end{equation}
We prove that the Weierstrass iterative function ${T \colon \mathcal{D} \subset \Kset^n \to \Kset^n}$ defined by 
\eqref{Weierstrass-iteration-function} is an iterated contraction with respect to $E$.

\medskip
We begin this section with a well-known result whose proof we include for completeness.

\begin{proposition} \label{prop:Lagrange-formula-Weierstrass}
Let ${f \in \Kset[z]}$ be a monic polynomial of degree ${n \ge 2}$, and let ${x \in \Kset^n}$ be a vector with distinct components. 
Then for all ${z \in \Kset}$,
\[
f(z) = \sum_{i=1}^{n} W_i(x) \prod_{j \ne i}{(z - x_j)} + \prod_{j=1}^{n} (z - x_j).
\]
\end{proposition}

\begin{proof}
Applying Lagrange's interpolation formula to the polynomial
\[
g(z) = f(z) - \prod_{j = 1}^n (z - x_j)
\]
at the nodes ${x_1,\ldots,x_n}$, we get the desired presentation of $f$.
\end{proof}

Using the Weierstrass correction $W_f$ one can state the following basic existence result for polynomial zeros. 

\begin{proposition}[Basic Existence Theorem] \label{prop:Basic-existence-theorem}
Let $f \in \Kset[z]$ be a polynomial of degree $n \ge 2$, and let
$(x^k)$ be an infinite sequence of vectors in $\Kset^n$ with distinct components satisfying the following two conditions:
\begin{enumerate}[(i)]
\item The sequence $(x^k)$ converges to a point $\xi$ in ${\Kset}^n$;
\item The sequence $(W_f(x^k))$ converges to the zero-vector in ${\Kset}^n$.
\end{enumerate}
Then $\xi$ is a root-vector of $f$.
\end{proposition}

\begin{proof}
Without lose of generality we may assume that $f$ is a monic polynomial. 
Applying Proposition~\ref{prop:Lagrange-formula-Weierstrass} with $x = x^k$, we obtain
\[
f(z) =
\sum_{i=1}^n {W_i(x^k) \prod_{j \neq i} {(z - x_j^k)}} + \prod_{j = 1}^n {(z - x_j^k)}.
\]
for all $k$. Passing to the limit when $k \rightarrow \infty$, we get
\(
f(z) = \prod_{j = 1}^n {(z - \xi_j)}
\)
which completes the proof.
\end{proof}

The following remarkable result for the Weierstrass iteration is an immediate consequence of Proposition~\ref{prop:Basic-existence-theorem}.

\begin{proposition} \label{prop:Weierstrass-iteration-existence-theorem}
Let $f \in \Kset[z]$ be a polynomial of degree $n \ge 2$. If for some initial guess $x^0 \in \Kset^n$, the Weierstrass iteration \eqref{eq:Weierstrass-iteration} is well-defined and converges to a vector $\xi \in \Kset^n$, then $\xi$ is a root-vector of $f$.
\end{proposition}

Let us give some historical historical notes about Proposition~\ref{prop:Weierstrass-iteration-existence-theorem}. 
In 1972, Petkov \cite[p.~272]{Pet74b} has briefly noted (without proof) that if the Weierstrass iteration 
\eqref{eq:Weierstrass-iteration} for a complex polynomial $f$ converges to a vector $\xi \in \Cset^n$, then 
$\xi$ is a root-vector of $f$. This result for complex polynomials with simple zeros has been proved by Zheng \cite{Zhe82} in 1982 and re-obtained in 1994 by Hopkins, Marshall, Schmidt and Zlobec \cite[Proposition 4.1]{HMSZ94}.

\begin{proposition} \label{prop:localization}
Let ${\Phi \colon D \subset \Kset^n \to \Kset^n}$ be a mapping defined on a set $D$ which contains only vectors with distinct components, 
and let ${E \colon D \to \Rset_+}$ be defined by
\[
E(x) = \left\| \frac{\Phi(x)}{d(x)} \right\|_p \qquad (1 \le p \le \infty).
\]
Let ${x \in D}$ be such that 
\( 
{E(x) \in J},
\)
where $J \subset \Rset$ is an interval containing zero. 
Furthermore, let there exist two function $\beta,\gamma \colon J \to \Rset_+$ such that
\begin{equation} \label{eq:gamma-beta}
\beta(t) < 1 - 2^{1/q} \, t \, \gamma(t)
\quad\text{for all}\quad t \in J.
\end{equation}
Then the closed disks 
\begin{equation}  \label{disks-disjoint}
D_i = \{z \in \Kset : |z - x_i| \le r_i \}, \quad i = 1,2,\ldots,n,
\end{equation}
where 
\[
r_i = \frac{\gamma(E(x))}{1 - \beta(E(x))} \, |\Phi_i(x)|, 
\]
are mutually disjoint.
\end{proposition}

\begin{proof}
For simplicity, we set  ${b = 2^{1/q}}$ and $C = \gamma(E(x)) / (1 - \beta(E(x)))$. 
It follows from \eqref{eq:gamma-beta} that
\[
 \frac{b \, t \, \gamma(t)}{1 - \beta(t)} < 1 \quad\text{for all}\quad t \in J.
\]
This implies that ${b \, C \, E(x) < 1}$ since ${E(x) \in J}$.
To prove that the disks \eqref{disks-disjoint} are mutually disjoint it is sufficient to show that
${r_i + r_j < |x_i - x_j|}$, that is, 
\begin{equation} \label{eq:distance-center-radius}
C \, (|\Phi_i(x)| + |\Phi_j(x)|) < |x_i - x_j|
\quad\text{for}\quad i \ne j.
\end{equation}
Suppose ${C > 0}$ since the case ${C = 0}$ is obvious. 
From the definition of $d(x)$, the power mean inequality $M_1 \le M_p$ and the inequality ${b \, C \, E(x) < 1}$, we obtain
\[
\frac{|\Phi_i(x)| + |\Phi_j(x)|}{|x_i - x_j|} \le \frac{|\Phi_i(x)|}{d_i(x)} + \frac{|\Phi_j(x)|}{d_j(x)} \le b \, E(x) < \frac{1}{C}
\] 
which yields \eqref{eq:distance-center-radius}.
\end{proof}

\begin{lemma} 
Let ${f \in \Kset[z]}$ be a polynomial of degree ${n \ge 2}$ which has $n$ simple zeros in $\Kset$, ${\xi \in \Kset^n}$ be a root-vector of $f$, ${x \in \Kset^n}$ and ${1 \le p \le \infty}$. Then
\[
|x_i - x_j| \ge (1 - 2^{1 / q} E(x)) \, d_j(\xi) \quad\text{and}\quad    |x_i - \xi_j| \ge (1 - E(x)) \, d_i(\xi)  
\]
for ${i \ne j}$, where ${E \colon \Kset^n \to \Rset_+}$ is defined by \eqref{eq:FIC1-SM}.
\end{lemma}

\begin{proof}
Setting in Proposition~\ref{prop:inequality-1} ${u = x}$ and ${v = \xi}$ and taking into account the definition of $d(\xi)$, we obtain the statement of the lemma.
\end{proof}

\begin{lemma} \label{lem:formula-Weierstrass-correction}
Let ${f \in \Kset[z]}$ be a polynomial of degree ${n \ge 2}$ and ${1 \le p \le \infty}$.
Suppose $x \in \Kset^n$ is a vector with distinct components satisfying
\begin{equation} \label{eq:FIC-SM}
\left\| \frac{W(x)}{d(x)} \right\|_p < \frac{1}{2^{1 / q}} \, .
\end{equation}
Then the vector $\hat{x} = x - W(x)$ has distinct components and
\begin{equation} \label{eq:formula-Weierstrass-correction}
W_i(\hat{x}) = (\hat{x}_i - x_i)
\sum_{j \neq i} {\frac{W_j(x)}{\hat{x}_i - x_j}}
\prod_{j \neq i}
{\left( 1 + \frac{\hat{x}_j - x_j}{\hat{x}_i - \hat{x}_j} \right)}
\qquad (i = 1, \ldots, n).
\end{equation}
\end{lemma}

\begin{proof}
Applying Proposition~\ref{prop:inequality-1} with $u = \hat{x}$ and $v = x$ and taking into account condition \eqref{eq:FIC-SM}, we conclude that 
$\hat{x}_i \neq \hat{x}_j$ and $\hat{x}_i \neq x_j$ for $i \neq j$. Hence, both sides of \eqref{eq:formula-Weierstrass-correction} are well-defined. In particular, the vector $\hat{x}$ has distinct components. 
Let ${i \in I_n}$ be fixed.
If $\hat{x}_i = x_i$, then \eqref{eq:formula-Weierstrass-correction} holds trivially. Assume that
$\hat{x}_i \neq x_i$.
It follows from Proposition~\ref{prop:Lagrange-formula-Weierstrass} that for every $z \in \Kset$ 
such that $z \neq x_j$ ($j = 1, \dots, n$),
\[
f(z) =
\left( 1 + \sum_{j = 1}^n {\frac{W_j(x)}{z - x_j}}   \right) \prod_{j = 1}^n {(z - x_j)}
\]
which can be rewritten in the form
\[
f(z) =
(z - x_i) \left( 1 + \frac{W_i(x)}{z - x_i} +\sum_{j \neq i} {\frac{W_j(x)}{z - x_j}}   \right) \prod_{j \neq i} {(z - x_j)}.
\]
Setting here $z = \hat{x}_i$ and taking into account that $\hat{x}_i = x_i - W_i(x)$, we get
\[
f(\hat{x}_i) =
(\hat{x}_i - x_i) \sum_{j \neq i} {\frac{W_j(x)}{\hat{x}_i - x_j}} \prod_{j \neq i} {(\hat{x}_i - x_j)}.
\]
Therefore,
\[
W_i(\hat{x}) =
\frac{f(\hat{x}_i)}{\displaystyle\prod_{j \neq i} {(\hat{x}_i - \hat{x}_j)}} = (\hat{x}_i - x_i)
\sum_{j \neq i} {\frac{W_j(x)}{\hat{x}_i - x_j}}
\prod_{j \neq i} {\frac{\hat{x}_i - x_j}{\hat{x}_i - \hat{x}_j}}
\]
which coincides with \eqref{eq:formula-Weierstrass-correction}.
\end{proof}

\begin{lemma} \label{lem:Weierstrass-semilocal-convergence}
Let ${f \in \Kset[z]}$ be a polynomial of degree ${n \ge 2}$ and ${1 \le p \le \infty}$.
Suppose ${x \in \Kset^n}$ is a vector with distinct components such that
\begin{equation} \label{eq:FIC-CS-E(x)}
E(x) = \left\| \frac{W(x)}{d(x)} \right\|_p < \frac{1}{2^{1 / q}} \, .
\end{equation}
Then the vector $Tx = x - W(x)$ has distinct components and
\begin{equation} \label{eq:Weierstrass-semilocal-convergence-iterated-contraction}
E(Tx) \le \varphi(E(x))
\quad\text{and}\quad
\|Tx - T^2x\| \preceq \beta(E(x)) \, \|x - Tx\|,
\end{equation}
where the real functions $\varphi$ and $\beta$ are defined by
\begin{equation} \label{eq:Weierstrass-semilocal-convergence-varphi}
\varphi(t) = \frac{(n - 1)^{1/q} \, t^2}{(1 - t)(1 - 2^{1/q} \, t)}
\left (1 + \frac{t}{(n - 1)^{1/p} \, (1 - 2^{1/q} \, t)} \right )^{n - 1} ,
\end{equation}
\begin{equation} \label{eq:Weierstrass-semilocal-convergence-beta}
\beta(t) = \frac{(n - 1)^{1/q} \, t}{1 - t}
\left (1 + \frac{t}{(n - 1)^{1/p} \, (1 - 2^{1/q} \, t)} \right )^{n - 1} .
\end{equation}
\end{lemma}

\begin{proof}
Proposition~\ref{prop:inequality-1}(i) with $u = Tx$ and $v = x$ yields
\begin{equation} \label{eq:d(Tx)-d(x)}
d(Tx) \succeq \psi(E(x)) \, d(x),
\end{equation}
where the real function $\psi$ is defined by
\begin{equation} \label{eq:Weierstrass-semilocal-convergence-psi}
\psi(t) = 1 - 2^{1 / q} \, t .
\end{equation}
Obviously, \eqref{eq:d(Tx)-d(x)} implies that $Tx$ has distinct components.
For the sake of simplicity, we use the following notations:
\[
\hat{x} = Tx, \quad a = (n - 1)^{1 / q}, \quad b = 2^{1 / q} \quad\text{and}\quad c = (n - 1)^{1 / p} .
\] 
It follows from Lemma~\ref{lem:formula-Weierstrass-correction} that
\begin{equation} \label{eq:W(Tx)-W(x)}
\|W(Tx)\| \preceq \sigma \mu \, \|W(x)\|,
\end{equation}
where
\[
\sigma = \max_{i \in I_n} \sum_{j \neq i} {\left| \frac{W_j(x)}{\hat{x}_i - x_j} \right|}
\quad\text{and}\quad
\mu = \max_{i \in I_n} \prod_{j \neq i} {\left( 1 + \left| \frac{W_j(x)}{\hat{x}_i - \hat{x}_j} \right| \right)}.
\]
Applying Proposition~\ref{prop:inequality-1} with  $u = \hat{x}$ and $v = x$, we get
\[
|\hat{x}_i - x_j| \ge (1 - E(x)) \, d_j(x)
\quad\text{and}\quad
|\hat{x}_i - \hat{x}_j| \ge (1 - b \, E(x)) \,  d_j(x).
\]
Using the power mean inequality ${M_1 \le M_p}$ and Proposition~\ref{prop:product-lemma}, we obtain
\begin{equation} \label{eq:sigma-mu-bounds}
\sigma \le \frac{a \, E(x)}{1 - E(x)}
\quad\text{and}\quad
\mu \le \left( 1 + \frac{E(x)}{c \, (1 - b \, E(x))} \right)^{n - 1}.
\end{equation}
Now from \eqref{eq:W(Tx)-W(x)} and \eqref{eq:sigma-mu-bounds}, we get
\begin{equation} \label{eq:Weierstrass-semilocal-convergence-iterated-contraction-W}
\|W(Tx)\| \preceq \beta(E(x)) \, \|W(x)\|
\end{equation}
which coincides with the second inequality in \eqref{eq:Weierstrass-semilocal-convergence-iterated-contraction} since  
${W(x) = x - Tx}$.
It follows from \eqref{eq:Weierstrass-semilocal-convergence-iterated-contraction-W} and \eqref{eq:d(Tx)-d(x)} that
\[
\left\| \frac{W(Tx)}{d(Tx)} \right\| \preceq
\frac{\beta(E(x))}{\psi(E(x))} \left\| \frac{W(x)}{d(x)} \right\|.
\]
Taking the $p$-norm here, we get the first inequality in \eqref{eq:Weierstrass-semilocal-convergence-iterated-contraction} 
since $\beta = \phi \, \psi$.
\end{proof}

Now we can state and prove our semilocal convergence theorem for the Weierstrass method which is the main result of this section.

\begin{theorem} 
\label{thm:semilocal-convergence-Weierstrass}
Let $\Kset$ be a complete normed field,
$f \in \Kset[z]$ be a polynomial of degree $\, n \ge 2$ and $1 \le p \le \infty$.
Suppose $x^0 \in \Kset^n$ is an initial guess with distinct components satisfying
\begin{equation} \label{eq:semilocal-convergence-Weierstrass-initial-conditions}
E(x^0) < 1 / 2^{1/q} \quad\text{and}\quad \phi(E(x^0)) \le 1 ,
\end{equation}
where the function $E$ is defined by \eqref{eq:FIC3-SM} and the real function $\phi$ is defined by
\begin{equation} \label{eq:semilocal-convergence-Weierstrass-phi}
\phi(t) = \frac{(n - 1)^{1/q} \, t}{(1 - t)(1 - 2^{1/q} \, t)}
\left (1 + \frac{t}{(n - 1)^{1/p} \, (1 - 2^{1/q} \, t)} \right )^{n - 1} .
\end{equation}
Then the following statements hold true.
\begin{enumerate}[(i)]
\item 
\textsc{Convergence}. Starting from $x^0$, the Weierstrass iteration \eqref{eq:Weierstrass-iteration} is well-defined, remains in the closed ball ${\overline{U}(x_0,\rho)}$ and converges to a root-vector $\xi$ of $f$, where 
\[
\rho =  \frac{\| W(x^0)\|}{1 - \beta(E(x^0))}
\]
and the real function $\beta$ is defined by \eqref{eq:Weierstrass-semilocal-convergence-beta}.
Besides, the convergence is quadratic provided that ${\phi(E(x^0)) < 1}$.
\item 
\textsc{A priori estimate}. For all $n \ge 0$ we have the estimate
\begin{equation} \label{eq:Semilocal-Convergence-Weierstrass-a-priori-estimate}
\|x^k - \xi\| \preceq \frac{\theta^k \, \lambda ^{2^k - 1}}{1 - \theta \, \lambda^{2^k}} \, \|x^1 - x^0\|,
\end{equation}
where $\lambda = \phi(E(x^0))$, $\theta  = \psi(E(x^0))$ and the real functions $\psi$ is defined by 
\eqref{eq:Weierstrass-semilocal-convergence-psi}.
\item 
\textsc{First a posteriori estimate}. For all $k \ge 0$ we have the following estimate
\begin{equation} \label{eq:semilocal-convergence-Weierstrass-a-posteriori-estimate-1}
\|x^k - \xi\| \preceq \frac{\|x^{k+1} - x^k\|}{1 - \beta(E(x^k))} \, .
\end{equation}
\item 
\textsc{Second a posteriori estimate}. For all $k \ge 0$ we have the following estimate
\begin{equation} \label{eq:semilocal-convergence-Weierstrass-a-posteriori-estimate-2}
\|x^{k+1} - \xi\| \preceq \frac{\theta_k \lambda_k}{1 - \theta_k (\lambda_k)^2} \, \|x^{k+1} - x^k\| ,
\end{equation}
where $\lambda_k = \phi(E(x^k))$, $\theta_k = \psi(E(x^k))$.
\item \textsc{Some other estimates}. For all $k \ge 0$ we have
\begin{equation} \label{eq:semilocal-convergence-Weierstrass-other-estimates-1}
\|x^{k+2} - x^{k+1}\| \preceq \theta \, \lambda^{2^k} \, \|x^{k+1} - x^k\|,
\end{equation}
\begin{equation} \label{eq:semilocal-convergence-Weierstrass-other-estimates-2}
\|x^{k+1} - x^k\| \preceq \theta^k \, \lambda^{2^k - 1} \, \|x^1 - x^0\|.
\end{equation}
\item \textsc{Localization of the zeros}.
If ${\phi(E(x^0 )) < 1}$, then $f$ has $n$ simple zeros in $\Kset$.
Moreover, for every ${k \ge 0}$ the closed disks
\begin{equation}  \label{disks-disjoint-Weierstrass}
D_i^k = \{z \in \Kset : |z-x_i^{k}| \le r_i^k \},
\quad i = 1,2,\ldots,n,
\end{equation}
where
\[
r_i^k = \frac{|W_i(x^k)|}{1 - \beta(E(x^k)} \, ,
\]
are mutually disjoint and each of them contains exactly one zero of $f$.	
\end{enumerate}
\end{theorem}

\begin{proof}
Denote by $R = R(n, p)$ the unique solution of the equation
$\phi(t) = 1$ in the interval $(0, 1 / 2^{1/q})$. Then the initial conditions 
\eqref{eq:semilocal-convergence-Weierstrass-initial-conditions} can be rewritten in the form $E(x^0) \le R$.
The function $\varphi$ defined by \eqref{eq:Weierstrass-semilocal-convergence-varphi} is quasi-homogeneous of the second degree on 
$[0, 1 / 2^{1/q})$ and $R$ is its fixed point. It follows from Proposition~\ref{prop:gauge-function-sufficient-condition} that $\varphi$
is a gauge function of the second order on $J = [0,R]$.
The function $\beta$ is increasing on $J$ satisfying $\beta = \phi \, \psi$ and
$\beta(R) = \psi(R) < 1$. Hence, the function $t \beta(t)$ is a strict gauge function of the second order on $J$.
Now we shall apply Corollary~\ref{cor:third-convergence-theorem-CMS} to the Weierstrass iteration function 
$T \colon \mathcal{D} \subset \Kset^n \to \Kset^n$.
From Lemma~\ref{lem:Weierstrass-semilocal-convergence} and Proposition~\ref{prop:initial-point-test}, we conclude that: 

$\bullet$ ${E \colon \mathcal{D} \to \Rset_+}$ is a functions of initial conditions of $T$ with gauge function $\varphi$ 
of order ${r = 2}$ on $J$.

$\bullet$ ${T \colon \mathcal{D} \to \Kset^n}$ is an iterated contraction with respect to $E$ with control function $\beta$. 

$\bullet$ Every vector ${x^0 \in \mathcal{D}}$ satisfying condition \eqref{eq:semilocal-convergence-Weierstrass-initial-conditions} is an initial point of $T$.

Now it follows Corollary~\ref{cor:third-convergence-theorem-CMS} that the Weierstrass iteration \eqref{eq:Weierstrass-iteration} 
is well-defined and converges to a vector $\xi \in \Kset^n$. 
The claims (i)-(v) follow immediately from Corollary~\ref{cor:third-convergence-theorem-CMS} and 
Proposition~\ref{prop:Weierstrass-iteration-existence-theorem}. 
It remains to prove the claim (vi).
Let ${E(x^0) < 1/2}$ and ${\phi(E(x^0)) < 1}$, that is, ${E(x^0) < R}$.
Then ${E(x^k) < R}$ since every $x^k$ is an initial point of $T$.
Taking into account that ${\beta (t) = \phi (t) \, \psi(t)}$ and ${\phi(t) < 1}$ on
${[0,R)}$, we conclude that ${\beta(t) < \psi(t)}$ for all ${t \in [0,R)}$. 
Therefore condition \eqref{eq:gamma-beta} holds with ${\gamma(t) \equiv 1}$.
Applying Proposition~\ref{prop:localization} with ${\Phi = W}$ we conclude that the 
disks \eqref{disks-disjoint-Weierstrass} are mutually disjoint. 
On the other hand, it follows from claim (iii) that each of these disks contains at least one zero of $f$.
Therefore, each of the disks \eqref{disks-disjoint-Weierstrass} contains exactly one zero of $f$. 
In particular, $f$ has $n$ simple zeros in $\Kset$.
This complete the proof of the theorem.
\end{proof}

It should be noted that recently Proinov and Petkova \cite{PP14a} proved another semilocal convergence theorem for 
the Weierstrass method under initial conditions involving the Vi\`ete operator.

\begin{remark} \label{rem:semilocal-convergence-Weierstrass-initial-conditions-2}
Let $R(n,p)$ be the unique solution of the equation ${\phi(t) = 1}$ in the interval $(0, 1/2^{1/q})$, where $\phi$ is defined by 
\eqref{eq:semilocal-convergence-Weierstrass-phi}. Then the initial conditions 
\eqref{eq:semilocal-convergence-Weierstrass-initial-conditions} of 
Theorem~\ref{thm:semilocal-convergence-Weierstrass} can also be written in the following equivalent form
\begin{equation} \label{eq:semilocal-convergence-Weierstrass-initial-conditions-2}
E(x^0) \le R(n,p),
\end{equation}
where ${E \colon \mathcal{D} \to \Rset_+}$ is defined by \eqref{eq:FIC3-SM}. 
According to Theorem~\ref{thm:semilocal-convergence-Weierstrass} the domain of quadratic convergence of the Weierstrass method is the set
\begin{equation} \label{eq:domain-of-quadratic-convergence-of-the-Weierstrass-method}
\mathscr{A} = \{ x \in \mathcal{D} : E(x) < R(n,p) \}.
\end{equation}
Below we show that this cannot be improved.
\end{remark}.

\begin{example}
In this example we show that if ${p = \infty}$, then the domain of quadratic convergence \eqref{eq:domain-of-quadratic-convergence-of-the-Weierstrass-method} of the Weierstrass method as well as the estimates
\eqref{eq:Semilocal-Convergence-Weierstrass-a-priori-estimate},
\eqref{eq:semilocal-convergence-Weierstrass-a-posteriori-estimate-1},
\eqref{eq:semilocal-convergence-Weierstrass-a-posteriori-estimate-2},
\eqref{eq:semilocal-convergence-Weierstrass-other-estimates-1},
\eqref{eq:semilocal-convergence-Weierstrass-other-estimates-2}
cannot be improved in the sense that for every polynomial ${f(z) = a_0 z^2 + a_1 z + a_2}$ in $\Kset[z]$ with zero discriminant, 
there exist infinitely many initial guesses ${x^0 \in \Kset^2}$ such that ${E(x^0) = R(2,\infty) = \frac{1}{4}}$ and the following 
two statements hold true:
\begin{enumerate}[(a)]
	\item The Weierstrass iteration \eqref{eq:Weierstrass-iteration} is well-defined and converges linearly to the root-vector 
	${\xi = (a,a)}$ of $f$, where ${a = - a_1/(2 a_0)}$.
	\item Each of the estimates 
	\eqref{eq:Semilocal-Convergence-Weierstrass-a-priori-estimate}-\eqref{eq:semilocal-convergence-Weierstrass-other-estimates-2}  
becomes equality. 
\end{enumerate}
 
To prove this, we choose a vector ${x^0 = (a + b,a - b)}$, where $b$ is an arbitrary element of $\Kset$.
It is easy to calculate that ${E(x^0) = R(2,\infty) = \frac{1}{4}}$, where $E$ is defined by \eqref{eq:FIC3-SM} with ${p = \infty}$. 
According to Theorem~\ref{thm:semilocal-convergence-Weierstrass}, the Weierstrass iteration \eqref{eq:Weierstrass-iteration} is 
well-defined and converges to a root-vector $\xi$ of $f$ with estimates
\eqref{eq:Semilocal-Convergence-Weierstrass-a-priori-estimate}-\eqref{eq:semilocal-convergence-Weierstrass-other-estimates-2},
which can be written in the following form:
\begin{eqnarray} \label{eq:exact-estimates}
\|x^k - \xi\| & \preceq & 2 \left( \frac{1}{2}\right)^k \|x^1 - x^0\|, \nonumber\\
\|x^k - \xi\| & \preceq & 2 \, \|x^{k+1} - x^k\|, \nonumber\\
\|x^{k} - \xi\| & \preceq & \|x^{k} - x^{k-1}\|,\\
\|x^{k+2} - x^{k+1}\| & \preceq & \frac{1}{2}  \, \|x^{k+1} - x^k\|, \nonumber\\
\|x^{k+1} - x^k\| & \preceq & \left( \frac{1}{2}\right)^k  \|x^1 - x^0 \nonumber\|.
\end{eqnarray}
On the other hand, the Weierstrass iteration \eqref{eq:Weierstrass-iteration} for ${f(z) = a_0 z^2 + a_1 z + a_2}$
with the initial guess ${x^0 = (a + b,a - b)}$ can be written in the form
\[
x^k = \left( a + \frac{b}{2^k},a - \frac{b}{2^k} \right).
\]
Hence, the sequence $(x^k)$ converges to the vector ${\xi = (a,a)}$. Besides,
\[
x^{k+1} -\xi = \frac{1}{2} \, (x^k - \xi). 
\]
This shows that the order of convergence $(x^k)$ is exactly one (with asymptotic constant $\frac{1}{2}$) which proves the statement (a).
Furthermore, it is easy to check that $(x^k)$ satisfies the following identities:
\begin{eqnarray*}
x^k - \xi  & = & 2 \left( \frac{1}{2}\right)^k (x^0 - x^1),\\
x^k - \xi & = & 2 \, (x^k - x^{k + 1}),\\
x^k - \xi & = & x^{k-1} - x^k),\\
x^{k+1} - x^{k+2} & = & \frac{1}{2} (x^k - x^{k+1}),\\
x^k - x^{k+1} & = & \left( \frac{1}{2}\right)^k  (x^0 - x^1).
\end{eqnarray*}
These identities show that all the inequalities \eqref{eq:exact-estimates} become equalities.
This proves the statement (b).
\qed
\end{example}

\begin{remark}
It follows from Theorem~\ref{thm:semilocal-convergence-Weierstrass} and Lemma~\ref{eq:formula-Weierstrass-correction} that if an initial guess $x^0 \in \Kset^n$ satisfies the initial conditions \eqref{eq:semilocal-convergence-Weierstrass-initial-conditions}  
for some $1 \le p \le \infty$, then the Weierstrass iteration
\eqref{eq:Weierstrass-iteration} is well-defined and can be presented in the following two-point form:
\[
x^1 = x^0 - W_f(x^0),
\]
\begin{equation} \label{eq:Weierstrass-iteration-two-step-method}
x_i^{k + 1} = x_i^{k} - (x_i^k - x_i^{k - 1}) \,
\sum_{j \neq i} {\frac{x_j^{k - 1} - x_j^k}{x_i^k - x_j^{k - 1}}} \,
\prod_{j \neq i} {\frac{x_i^k - x_j^{k - 1}}{x_i^k - x_j^k}}
\qquad (i = 1, \ldots, n)
\end{equation}
\[
k =1, 2, \ldots
\]
Let us note that in the Weierstrass method \eqref{eq:Weierstrass-iteration}, we need to compute the polynomial values $f(x_i^k)$ for each iterative step. Under the initial conditions \eqref{eq:semilocal-convergence-Weierstrass-initial-conditions} the two-point iterative method \eqref{eq:Weierstrass-iteration-two-step-method} is equivalent to the Weierstrass method 
\eqref{eq:Weierstrass-iteration} but no longer needs to evaluate $f(x_i^k)$ for each iteration after the first step.
The two-step iterative process \eqref{eq:Weierstrass-iteration-two-step-method} was first presented in 1964 by 
Dochev and Byrnev \cite{DB64} in a slightly different form (see also Zheng \cite{Zhe82} and Yao \cite{Yao00}).
\end{remark}

\begin{remark}
In an earlier work \cite{Pro06b} (see also \cite{Pro06c}) we have stated without proof a weaker version of Theorem~\ref{thm:semilocal-convergence-Weierstrass}. In \cite{Pro06b} we give a detailed comparison of our old theorem with previous results.  
It should be noted that all corollaries given in \cite{Pro06b} can be improved using Theorem~\ref{thm:semilocal-convergence-Weierstrass} instead of Theorem~1 of \cite{Pro06b}. We end this section with a result which improves Corollary~2 of \cite{Pro06b}.
This result generalizes and improves the results of \cite {Pre80,WZ95,PCT95,Pet96,PHI98,Bat98,Han00}.     
\end{remark} 

\begin{corollary} \label{cor:semilocal-convergence-Weierstrass-2}
Let $\Kset$ be a complete normed field,
$f \in \Kset[z]$ be a polynomial of degree $n \ge 2$ and $1 \le p \le \infty$.
Suppose $x^0 \in \Kset^n$ is an initial guess with distinct components satisfying
\[
\left\| \frac{W(x^0)}{d(x^0)} \right\|_p \le \frac{1}{2(n - 1)^{1/q} + 2} .
\]
Then all conclusions \emph{(i)-(v)} of Theorem~\ref{thm:semilocal-convergence-Weierstrass} hold true.
Moreover, if $n \ge 3$, then $f$ has $n$ simple zeros in $\Kset$, 
the Weierstrass iteration \eqref{eq:Weierstrass-iteration} converges quadratically to $\xi$ 
and for every ${k \ge 0}$ the closed disks \eqref{disks-disjoint-Weierstrass} are mutually disjoint and each of them contains exactly one zero of $f$. 
\end{corollary}

%
%

\end{document}